\documentclass[12pt]{article}
\usepackage{amssymb}
\usepackage{amsmath}
\usepackage{amsthm}
\usepackage{graphicx}
\usepackage{typearea}
\usepackage{cite}
\usepackage{enumerate}
\usepackage{anyfontsize}

\usepackage{float}

\usepackage{enumitem}  
\usepackage{calc} 

\setlist{labelindent=1pt,itemsep=.5em}
\setlist[itemize]{leftmargin=1.2cm}
\setlist[enumerate]{itemindent=0em,leftmargin=1.2cm}
\setlist[enumerate,1]{label={\upshape(\roman*)}}

\usepackage{caption}

\usepackage[dvipsnames]{xcolor}
\usepackage{makeidx}         
\usepackage{graphicx}
\usepackage{multicol}
\usepackage{nicematrix}
\usepackage[bottom]{footmisc}
\makeindex
\usepackage[utf8]{inputenc}
\usepackage[T1]{fontenc}
\usepackage[english]{babel}
\usepackage{marginnote}
\usepackage{tikz-cd}

\usepackage{hyperref}
\hypersetup{
	colorlinks   = true,
	citecolor    = black,
	linkcolor =black,
	urlcolor=black,
}

\usepackage[capitalise]{cleveref} 

\newtheorem{theorem}{Theorem}
\newtheorem{lemma}[theorem]{Lemma}
\newtheorem{definition}[theorem]{Definition}
\newtheorem{proposition}[theorem]{Proposition}
\newtheorem{corollary}[theorem]{Corollary}
\newtheorem{example}[theorem]{Example}

\newtheorem{remark}[theorem]{Remark}

\newcounter{propertycounter}
\newcounter{propertycounterright}
\begin{document}
		
		\title{Hom-unitality and hom-associative structures}
		
		\author{
			Germ\'{a}n Garc\'{i}a Butenegro \footnote{E-mail: gegarcb@gmail.com}\, ,
			Abdennour Kitouni \footnote{E-mail: abdennour.kitouni@gmail.com}\, ,
			Sergei Silvestrov\footnote{ E-mail: sergei.silvestrov@mdu.se
				\text{(Corresponding author)}}
			\\ \\
			{\fontsize{12}{10.8}\normalshape
				\footnotemark[3] {} Department of Business and Mathematics,} \\
			{\fontsize{12}{10.8}\normalshape
				Faculty of Philosophy,} \\
			{\fontsize{12}{10.8}\normalshape
				M{\"a}lardalen University,} \\
			{\fontsize{12}{10.8}\normalshape Box 883, 72123 V\"{a}ster{\aa}s, Sweden}
		}
		
		\date{}
		
		
		
		\maketitle
		
		\abstract{
We study hom-associative structures on general, possibly non-associative algebras, focusing on one-sided and two-sided unital algebras. New characterizations and aspects of these structures, along with some important subclasses, are explored for nonassociative algebras. By exploiting the observation that the twisting linear map in the hom-associativity axiom of one-sided unital hom-associative algebras is a left or right multiplication operator by an element of the algebra (obtained by the action of the twisting map on a corresponding one-sided unity), a new characterization of the multiplicative twisting operators (or, in other words, the multiplicative hom-associative algebras) is established for one-sided unital algebras. This demonstrates a strong connection between multiplicativity in hom-associative structures and the idempotents of the algebra, thereby further enhancing our understanding of the structure and special nature of multiplicative hom-associative algebra structures as a special subclass of arbitrary hom-associative structures with arbitrary linear twisting maps.
Furthermore, new insights into subspaces and a subalgebra of hom-unities, that is, elements that induce hom-associativity by multiplication, are obtained. Moreover, because non-unital hom-associative algebras need not be twisted by a multiplication operator, a unitalization process is employed to describe a subalgebra of two-sided hom-unities that induce hom-associative structures on such algebras. This is formulated in terms of the eigenspaces of multiplication operators within the algebra. Additionally, the obtained insights and general results about the structure and characterization of hom-algebra structures are applied to some important known general classes of non-associative algebras, such as commutative, possibly non-associative algebras, Cayley-Dickson algebras, Leibniz algebras, and hom-Leibniz algebras.
		}
		
		
		
		\tableofcontents

		\section{Introduction}
		Hom-algebra structures provide a unified framework for the joint study of seemingly unrelated classes of algebraic structures by relating their defining axioms through linear twisting maps (unary operations) that deform these axioms in specific and non-trivial ways. Such deformations establish connections between the axioms of various related but distinct known and new classes of algebraic structures, enabling their simultaneous consideration and classification and uncovering common or distinct properties and relationships between their substructures and morphisms. In particular, hom-algebra structures of a given type include their classical non-twisted counterparts for special choices of twisting maps in the defining identities, creating additional opportunities for deformations, extensions of cohomological structures, and representations.
		Over the last two decades, following pioneering works \cite{Hartwig-Larsson-Silvestrov,LarssonSilvJA2005:QuasiHomLieCentExt2cocyid,Larsson-Silvestrov-QuasiLiealg,LSGradedquasiLiealg,LarssonSilvestrov-quasidefal2twderiv,ms:homstructure,MakSil-formaldefs}, hom-algebra structures have evolved into a broad and fruitful area of mathematics, exhibiting numerous connections to fundamental non-commutative and non-associative structures and concepts in non-commutative geometry, as well as offering possibilities for a unified treatment and applications to models in mathematical physics based on various important non-commutative or potentially non-associative structures.
		
		The area of hom-algebra structures originated with the work of Hartwig, Larsson, and Silvestrov \cite{Hartwig-Larsson-Silvestrov} in 2003, where general hom-Lie algebras and the more general quasi-hom-Lie algebras were introduced, together with a general method for constructing quasi-deformations of Witt and Virasoro algebras based on the discretization of vector fields via twisted derivations ($\sigma$-derivations). This approach was motivated by the discrete modifications of differential calculus and the $q$-deformed Jacobi identities observed for specific $q$-deformed algebras and their representations in physics, as well as in $q$-deformed differential calculus and homological algebra. Following \cite{Hartwig-Larsson-Silvestrov}, Larsson and Silvestrov carried out a systematic study of central extensions of quasi-hom-Lie algebras in \cite{LarssonSilvJA2005:QuasiHomLieCentExt2cocyid}, first appearing in 2004. At the same time, Larsson and Silvestrov introduced more general quasi-Lie algebras and quasi-Leibniz algebras in \cite{Larsson-Silvestrov-QuasiLiealg} (2004), and general color quasi-Lie algebras ($\Gamma$-graded $\epsilon$-quasi-Lie algebras) in \cite{LSGradedquasiLiealg} (2005), including color quasi-hom-Lie algebras and color hom-Lie algebras, quasi-Lie superalgebras, quasi-hom-Lie superalgebras, hom-Lie superalgebras, and hom-Lie algebras. Sigurdsson and Silvestrov continued this work in \cite{Sigurdsson-Silvestrov-Czech:witt} (2006) and \cite{SigSilvGLTbdSpringer2009-witt-centrext} (2009), where color quasi-Lie algebras of Witt type and their central extensions were studied. In a further fundamental development, important hom-Lie admissible algebras were discovered by Makhlouf and Silvestrov in \cite{ms:homstructure} (2006). In particular, \cite{ms:homstructure} introduced hom-associative algebras and showed that they are hom-Lie admissible, meaning that they yield hom-Lie algebras via the commutator map as a new product, thus constituting a natural generalization of associative algebras as Lie-admissible algebras leading to Lie algebras through the commutator map. Moreover, \cite{ms:homstructure} describes several other interesting classes of hom-Lie admissible algebras extending certain non-associative algebras to more general hom-algebra counterparts, as well as examples of finite-dimensional hom-Lie algebras. In \cite{MakSil-formaldefs} (2007), the study of formal deformations and cohomology of hom-algebras was initiated, giving rise to numerous significant developments in cohomology and formal deformation theory for various hom-algebra structures over the past two decades.
		
		Hom-associative algebras, introduced in \cite{ms:homstructure}, are the first examples of \emph{hom-Lie admissible} algebras, where the commutator induces a hom-Lie algebra. They have been extensively studied, including deformations, extensions, and Hochschild cohomology in the multiplicative case \cite{MakSil-formaldefs,Gohr:HAsurj,HurMak:HAcohomology,ArfFraMak:morphisms} and low-dimensional classifications \cite{MakZah:HAclassification}. The main construction method is \emph{twisting} the product of an associative algebra by an algebra homomorphism \cite{Yau:HAhomology}, a process applicable to various algebraic structures, including alternative and Jordan algebras \cite{Mak:HomAltHomJordan}. While many hom-associative algebras arise from twists, non-twist examples exist even in dimension three \cite{MakZah:HAclassification,GarButSil:ZeroDivision}. Other notable constructions leading to many non-twist examples include the hom-association process and hom-associative Ore extensions \cite{BackRic:HAWeyl,BackRicSil:HAOreWeakUnit}. Hom-algebras in which the twisting map is an algebra homomorphism are called \emph{multiplicative} hom-algebras. Hom-algebras from this special subclass possess many additional cohomological and divisibility properties \cite{Yau:HAhomology,AmmEjbMak:cohomology,ArfFraMak:morphisms,Mak:HomAltHomJordan,GarButKitSil:divisibility}, and the hom-associative ones are classified in low dimensions \cite{MakZah:HAclassification}. Numerous examples of hom-algebras are multiplicative, as they can be constructed by twisting \cite{Yau:HAhomology} or by employing algebra morphisms in other ways \cite{BackRic:HAWeyl,BackRicSil:HAOreWeakUnit}.
		Although not strictly necessary, multiplicativity often facilitates the study of cohomology and representations \cite{AmmEjbMak:cohomology,ArfFraMak:morphisms,Mak:HomAltHomJordan}. Simultaneously, multiplicative hom-algebras form a rather special subclass of hom-algebras, as multiplicativity is generally a restrictive condition. Any insight into the role and consequences of multiplicativity or non-multiplicativity of various classes of hom-algebra structures and the classifications or characterizations of multiplicative or non-multiplicative algebra structures is one of the important fundamental problems in the investigation, classification, construction, and application of various hom-algebra structures. It is deeply concerned with the interplay of multiplicativity of the twisting maps with the defining axioms of the class of hom-algebra structures viewed as special type of
		deformations or mutations of the axioms of the original (untwisted) class of algebras.
		
		Unitality in hom-associative algebras \cite{FreGohr:UnitCond,FreGohrSil:UnitalSurj} can be seen as strengthening hom-associativity so that it implies or nearly implies classical associativity \cite{FreGohr:UnitCond}. This observation motivates the introduction of \emph{hom-unities} \cite{FreGohrSil:UnitalSurj}, also referred to as \emph{weak unities} \cite{BackRicSil:HAOreWeakUnit,BackRic:HAWeyl}, which are elements $\omega$ satisfying the hom-unitality condition $\omega \cdot x = \alpha(x)$. Hom-algebras equipped with a hom-unity have previously been called \emph{unital} \cite{ms:homstructure,MakZah:HAclassification} or \emph{weakly unital} \cite{BackRicSil:HAOreWeakUnit,FreGohr:UnitCond}. For consistency with our previous work \cite{GarButSil:ZeroDivision}, we adopt the terminology \emph{hom-unital}, interpreting the prefix ``hom-X'' as indicating that the axiom $X$ is modified by a twisting map. The notion of hom-unity provides a weakened form of unitality adapted to the twisted setting of hom-algebras.
		
		While classical unitality ensures that multiplication interacts with the identity element in a strict manner, hom-unitality relaxes this requirement by incorporating the twisting map, making it compatible with the hom-associative structure. Hom-unities play an important role in the structural analysis and construction of hom-associative algebras. They often arise naturally in deformation processes, where classical identities are replaced by twisted analogs to maintain consistency with the underlying hom-structure. For example, in Ore-type extensions and hom-analogs of Weyl algebras \cite{BackRicSil:HAOreWeakUnit,BackRic:HAWeyl}, the existence of a hom-unity ensures that the extended algebra retains a generalized identity element compatible with the twisting map. Similarly, in the classification results for hom-associative algebras \cite{MakZah:HAclassification}, hom-unities provide a useful criterion for distinguishing between strongly and weakly twisted structures. Beyond theoretical interest, these constructions have applications in quantum deformations and noncommutative geometry, where hom-type identities model symmetries that are no longer strictly associative but remain coherent under twisting.
		
		The goal of this article is to study in depth the connection between unitality and hom-unitality in arbitrary algebras and establish hom-associativity conditions in terms of one-sided unities, hom-unities, and multiplication operators within the algebra.

        The constructions and results of this article aim to approach this problem in the full generality of arbitrary non-associative algebras.

        We generalize the associativity conditions for hom-associative algebras established by Fr\'{e}gier, Gohr, and Silvestrov \cite{FreGohrSil:UnitalSurj}. In this paper, associativity conditions are derived from two-sided unitality. We provide conditions for left- and right-unital algebras, which imply those in \cite{FreGohrSil:UnitalSurj} in the two-sided unital case. Fr\'{e}gier and Gohr \cite{FreGohr:UnitCond} used the nucleus and ideals to describe a subalgebra of a given two-sided unital algebra that encodes all hom-associative structures of the algebra. We describe this subalgebra explicitly in terms of commutation, association, and zero division relations, that is, via the center, nucleus, and annihilators. We extend their results to one-sided unital algebras by describing a subspace of one-sided multiplication operators that contains a subalgebra encoding all hom-associative structures of a given one-sided unital algebra. Furthermore, we examine the problem in non-unital algebras by describing a subspace encoding hom-associative structures given by multiplication operators and use a unitalization process to describe a subalgebra encoding hom-associative structures given by multiplication by central elements.

        \cref{sec-basicnotions} introduces hom-associative and unital algebras, together with the notations and conventions assumed across this article. \cref{sec:unital-1SU} studies properties of the twisting map that are extracted from the unitality axiom, including that it must be a multiplication operator; \cref{ssec:association-relations} studies the relation between associator and hom-associator of these algebras and establishes rules for invariance of relative nuclei in terms of the unity (\cref{prop:AssociatorsHomAssociators}); in \cref{ssec:multiplicativity} we show that any appropriate (or \textit{compatible}) twisting map is multiplicative if and only if the multiplier (or \textit{hom-unity}) is idempotent (\cref{thm:unital_multiplicativity}). In two-sided unital algebras, Fr\'{e}gier and Gohr \cite{FreGohr:UnitCond} prove that the multipliers of all compatible maps form a subalgebra; in \cref{sec:hom-unities-2SU}, we compute it explicitly in terms of the center, nucleus, and annihilators of the algebra (see \cref{thm:AC_two_sided_unital}). We illustrate this by computing the hom-associative structures in the first Cayley-Dickson algebras. \cref{sec:hom-unities-1SU} explores characterization of hom-associative structures in one-sided unital algebras. We describe the multipliers of compatible maps as sum of an element commuting with the unity and a global zero divisor (\cref{thm:ACL_split}). We prove that the subspace of all multipliers contains a subalgebra (\cref{thm:ACL1_subalgebra}) consisting of those multipliers commuting with unity. This subalgebra is then shown to be in bijection with all compatible maps, and a subset is described, consisting of the idempotents of this subalgebra, which encodes those compatible maps that are also multiplicative (\cref{thm:bijection_left_unital}). The relations between hom-unities are expressed in terms of commutation, association, and zero division (see Lemma \ref {lem:ACL1_properties_A}) and are used to prove that there are no unital hom-associative domains or division algebras (see Theorem \ref {thm:no_domains_L}). In \cref{ssec:comparing}, we examine this construction in terms of the two-sided version by Fr\'{e}gier and Gohr \cite{FreGohr:UnitCond} and explore hom-associative structures in one-sided unital hom-Leibniz algebras. In particular, it turns out that in the presence of hom-associativity, a unity on one side is incompatible with the hom-Leibniz identity on the same side, but compatible with that on the other side (\cref{prop:LeibnizHomCrossed}).  In \cref{sec:non-unital}, we describe different subspaces of hom-unities for non-unital algebras (\cref{prop:NU_subsets}) and show that, under certain zero division relations, associative algebras are hom-associative with a twisting map given by multiplication by central elements (\cref{prop:HUnA_assoc_no_ann}). In \cref{ssec:unitalizations}, we use a standard unitalization process to describe a subspace of two-sided hom-unities in any non-unital algebras. We illustrate this idea on Leibniz algebras and find a subspace of hom-unities hidden within the nilpotents of the algebra (\cref{prop:LeibHomUnities}); furthermore, we examine non-unital hom-Leibniz algebras, where we prove that a certain twist yields a hom-Lie algebra (\cref{prop:HLeibYauTwist}).
		
		\section{Basic notions, notations and conventions} \label{sec-basicnotions}
		An algebra over a field $\mathbb{F}$  is a linear space $A$ over $\mathbb{F}$ with a bilinear multiplication operation called a product or multiplication.
        The notations for the products, such as $\mu\colon A\times A\rightarrow A$, $(x,y)\mapsto x\cdot y$, are often simplified to juxtaposition $xy$ when no confusion arises. The algebra is \textit{associative} if $x\cdot (y\cdot z)=(x\cdot y)\cdot z$ for all $x,y,z\in A$, and it is \textit{commutative} if $x\cdot y=y\cdot x$ for all $x,y\in A$.
		If no confusion is possible, we identify the algebra with the associated linear space $A$. An element $e\in A$ is \textit{idempotent} if $e\cdot e=e$. Let $ E (A) $ denote the set of all idempotents is denoted $E(A)$.
		Commuting elements in noncommutative algebras are studied in terms of commutators.
		The commutator is the bilinear operator on an algebra $A$
		defined by $[x,y]_{-}=x\cdot y-y\cdot x$. Algebras where the commutator is always $0$ are \textit{commutative}.
		The zeros of the commutator indicate how wide the commutativity relation is in an algebra or a certain subset of it: the \textit{center} $Z(A)=\{a \mid [a,x]_{-}=0  \  \forall x\in A\}$ of an algebra $A$ is the subspace of elements of $A$ that commute with all its elements, and generally, the \textit{centralizer} of a subset $S\subseteq A$ is the subspace of elements of $A$ that commute with all elements of $S$, that is, $[a,s]_{-}=0$ for all $s\in S$.

        In any algebra, the commutator map $[,]_{-}:A\times A\rightarrow A$ is bilinear and skew-symmetric, $[x,y]_{-}=-[x,y]_{-}$, and hence the underlying linear space of any algebra is a skew-symmetric algebra with the commutator map as the new product.
        In general, for non-associative algebras, it is a rare and highly restrictive condition for the commutator operation to be associative on the entire algebra.
         For example, in any Lie algebra $(A,[\cdot,\cdot])$, by skew-symmetry of $[\cdot,\cdot]$, the Jacobi identity can be rewritten as $[[x,y],z]-[x,[y,z]]=[[x,z],y]$, implying thus that $(A,[\cdot,\cdot])$ is associative if and only if $[A,A]\subseteq Ann_A(A)=\{a\mid [a,x]=[x,a]=0 \ \forall x\in A\}$, that is, if and only if $(A,[\cdot,\cdot])$ is a nilpotent algebra of degree $3$ meaning that the products of any three elements are zero, since
         $[A,A]\subseteq Ann_A(A)\Leftrightarrow [[A,A],A]=0,\ [A,[A,A]]=0$ by the skew-symmetry of $[\cdot,\cdot]$. This shows that, for Lie algebras, as a class of not necessarily associative algebras, associativity is a rare and highly restrictive condition. The annihilators of a subset $S\subseteq A$ are the subspaces $Ann_A^l(S)=\{a\in A\mid a\cdot s=0\ \forall s\in S\}$ (left annihilator) and $Ann_A^r(S)=\{a\in A\mid s\cdot a=0\ \forall s\in S\}$ (right annihilator), and the subspace $Ann_A(A)$ is known as the annihilator of $A$. In any skew-symmetric algebra $(A,[\cdot,\cdot])$, $[x,y]_{-}=[x,y]-[y,x]=[x,y]+[x,y]=2[x,y]$ by skew-symmetry, and hence $[x,y]_{-}=0 \Leftrightarrow [x,y]=0$, that is, $Ann_A(A)=Z(A)$ in any skew-symmetric algebra over a field of characteristic different from $2$.
         The commutator is often considered the simplest example of a bilinear product that is not necessarily associative in associative algebras. The commutator product on any associative algebra $(A,\cdot)$ not only satisfies skew-symmetry but also  obeys the Jacobi identity $[x,[y,z]_{-}]_{-}+[y,[z,x]_{-}]_{-}+[z,[x,y]_{-}]_{-}=0$ for all $x,y,z\in A$, yielding thus a Lie algebra
        $(A,[\cdot,\cdot]_{-})$. This Lie algebra is associative if and only if it is nilpotent of degree $3$, meaning that
        $[[A,A]_{-},A]_{-}=0,\ [A_{-},[A,A]_{-}]_{-}=0$, which is equivalent to
        $[x,[y,z]_{-}]_{-}=x\cdot y \cdot z-y\cdot x\cdot z-z\cdot x\cdot y+z\cdot y\cdot x=0$ for all $x,y,z\in A$. This identity holds for all commutative associative algebras, but is clearly quite restrictive for non-commutative associative algebras, singling out a special subclass of non-commutative associative algebras that satisfy this special identity.

        The associator on an algebra $A$ is a trilinear operator defined by $[x,y,z]_{\mathrm{as}}=(x\cdot y)\cdot z-x\cdot (y\cdot z)$.
		The associator can be seen as the commutator of multiplication operators $L_x$ and $R_z$ in the following sense:
		\begin{multline*}
			[x,y,z]_{\mathrm{as}}=(x\cdot y)\cdot z-x\cdot (y\cdot z)\\
			=R_z(x\cdot y)-L_x(y\cdot z)=R_z(L_x(y))-L_x(R_z(y))=[R_z,L_x]_{-}(y).
		\end{multline*}
		The zeroes of the associator correspond to associating elements of the algebra.
		\begin{definition}
			The \textit{relative nucleus} of $S\subseteq A$ is the subspace of  $A$ consisting of elements that associate with all elements in $S$, that is,
			\begin{align*}
				N_A(S) &=\{a\in A \mid [a,x,y]_{\mathrm{as}}=[x,a,y]_{\mathrm{as}}=[x,y,a]_{\mathrm{as}}=0\ \forall x,y\in S\} \\
				&= \{a\in A \mid [a,S,S]_{\mathrm{as}}= [S,a,S]_{\mathrm{as}}=[S,S,a]_{\mathrm{as}}=0\}.
			\end{align*}
			
			The left, right and middle relative nuclei of $S\subseteq A$ are the subspaces of $A$ consisting of elements that associate with all $S$ in the corresponding position:
			\begin{align*}
				N_A^l(S) &=\{a\in A \mid \left[a,s,t\right]=0\ \forall s,t\in S\} \\
				N_A^m(S) &=\{a\in A \mid \left[s,a,t\right]=0\ \forall s,t\in S\} \\
				N_A^r(S) &=\{a\in A \mid \left[s,t,a\right]=0\ \forall s,t\in S\}
			\end{align*}
			If $S=A$, we denote them $N(A), N^l(A), N^m(A)$ and $N^r(A)$.
		\end{definition}
		Hom-algebras appear as natural deformations of usual algebras by the action of linear maps. A hom-algebra is an algebra equipped with a linear map $\alpha\colon A\rightarrow A$ that deforms the fundamental identities of the algebra.

		\begin{definition}[\!\!{\cite[Definition 1.1]{ms:homstructure}}]
			A hom-associative algebra is a triple  $(A,\mu,\alpha)$ where $A$ is a linear space, $\mu\colon A\times A\rightarrow A$ is a bilinear product operation denoted by dot or juxtaposition, and $\alpha\colon A\rightarrow A$ is a linear map such that
			\[
			\alpha(x)\cdot(y\cdot z)=(x\cdot y)\cdot \alpha(z)
			\]
			for all $x,y,z\in A$. If no confusion is possible, we identify a hom-algebra with the pair $(A,\alpha)$. The twisting map $\alpha$ is called trivial if $\alpha=\lambda \mathrm{id}_A,\ \lambda\in\mathbb{F}$ and non-trivial otherwise.
		\end{definition}
		
		A hom-associative algebra is \textit{multiplicative} if  $\alpha(x\cdot y)=\alpha(x)\cdot \alpha(y)$ for all $x,y\in A$. A hom-associative algebra where $\alpha=\mathrm{id}_A$ is associative. A common way to obtain hom-associative algebras is to twist an associative algebra by a homomorphism \cite[Theorem 2.3]{Yau:HAhomology}. If $(A,\mu)$ is associative, then $(A,\alpha\circ\mu,\alpha)$ is hom-associative.
		\begin{example}
			Consider $\mathbb{C}$ as a linear space over $\mathbb{R}$, $x\ast y=\overline{x\cdot y}$ and $\alpha\colon x\mapsto \Bar{x}$, where $\Bar{x}$ is the complex conjugate of $x$. The triple $(\mathbb{C},\ast',\alpha)$ is a hom-associative algebra. Indeed, for all $x,y,z\in\mathbb{C},$
			\begin{align*}
				\alpha(x)\ast'(y\ast' z) &=\overline{\alpha(x)\cdot(y\ast' z)} =\overline{\Bar{x}\cdot \overline{{y\cdot z}}}=x\cdot y\cdot z
				\\
				& =\overline{\overline{x\cdot y}\cdot{\Bar{z}}}  =\overline{(x\ast' y)\cdot\alpha(z)}=(x\ast' y)\ast'\alpha(z).
			\end{align*}
		\end{example}
		
		This procedure is compatible with more general algebras and linear maps: the hom-associator of a twisted algebra $(A,\alpha\circ\mu,\alpha)$ can be expressed in terms of the original product as $[x,y,z]_{\mathrm{as}}^\alpha=\alpha(\alpha(x\cdot y)\cdot \alpha(z)-\alpha(x)\cdot \alpha(y\cdot z))$.
		\begin{theorem}
			Let $(A,\mu)$ be an algebra and $\alpha\colon A\rightarrow A$ be a linear map. Then $(A,\alpha\circ\mu,\alpha)$ is hom-associative if and only if, for all $x,y,z\in A$, $\alpha(\alpha(x\cdot y)\cdot \alpha(z)-\alpha(x)\cdot \alpha(y\cdot z))=0$.
		\end{theorem}
		
		\begin{example}
			Let $(\mathbb{C}[t],\mu)$ denote complex polynomials in $t$ with the usual product and the multiplication operator $L_t\colon p(t)\rightarrow tp(t)$. For any three generators $t^n,\ t^m,\ t^l$, the hom-associator with respect to $L_t$ is
			\begin{multline*}
				L_t(L_t(t^n\cdot t^m)\cdot L_t(t^l)-L_t(t^n)L_t(t^m\cdot t^l))=L_t(t^{n+m+1}\cdot t^{l+1}-t^{n+1}\cdot t^{m+l+1})\\
				=L_t(t^{n+m+l+2}-t^{n+m+l+2})=0.
			\end{multline*}
			Thus, $(\mathbb{C}[t],L_t\circ\mu,L_t)$ is hom-associative.
		\end{example}
		
		The hom-associator is a natural extension of the associator to hom-algebras, twisting or deforming associators by a twisting map.

		\begin{definition}[\!\!{\cite[Definition 2.5]{ms:homstructure}}]
			The hom-associator $[\cdot,\cdot,\cdot]_{\mathrm{as}}^\alpha\colon A\times A\times A\rightarrow A$ is the trilinear operator defined by $[x,y,z]_{\mathrm{as}}^\alpha=(x\cdot y)\cdot \alpha(z)-\alpha(x)\cdot (y\cdot z)$.
		\end{definition}
		
		Similar to the associator in associative algebras, an algebra is hom-associative if and only if the corresponding \textit{hom-associator} is always zero.
		
		A unity of a (hom-)algebra $(A,\alpha)$ is an element $1_\star\in A$ such that, for all $x\in A$, $1_\star\cdot x=x$ (left unity), or $x\cdot 1_\star=x$ (right unity), or both (two-sided unity). A hom-algebra equipped with unity is known as unital and denoted $(A,\alpha,1_\star)$.
		
		Two-sided unital hom-associative algebras are very close to being associative; it suffices that the twisting map is injective or surjective \cite{FreGohrSil:UnitalSurj}. In particular, the quotient algebra $A/\ker\alpha$ is associative. Therefore, it is necessary to introduce a generalized notion of unity for hom-algebras.
		
		\begin{definition}
			A hom-unity or weak unity of a hom-algebra $(A,\alpha)$ is an element $\omega_\star\in A$ such that, for all $x\in A$, $\omega_\star\cdot x=\alpha(x)$ (left hom-unity), $x\cdot \omega_\star=\alpha(x)$ (right hom-unity), or both (two-sided hom-unity). A hom-algebra with a hom-unity is known as hom-unital or weakly unital and is denoted $(A,\alpha,\omega_\star)$.
			A hom-algebra being left or right hom-unital is equivalent to $\alpha=L_{\omega_\star}$ or  $\alpha=R_{\omega_\star}$ respectively.
		\end{definition}
		\begin{example}[\!\!\cite{FreGohr:UnitCond}]
			The twist $(A,\alpha\circ\mu,\alpha,1_\star)$ of a two-sided unital algebra $(A,\mu,1_\star)$ by a homomorphism $\alpha$ is two-sided hom-unital. Indeed, for all $x\in A$,
			\begin{equation*}
				(\alpha\circ\mu)(1_\star, x)=\alpha(1_\star\cdot x)=\alpha(x)=\alpha(x\cdot 1_\star)=(\alpha\circ\mu)(x,1_\star).
			\end{equation*}
		\end{example}
		
		In this article, $L_a$ and $R_a$ denote, respectively, left and right multiplication by $a\in A$.  The naming convention  for these algebras is presented in Table \ref{tab:naming_convention}.
		
		\begin{table}[H]
			\makebox[\textwidth][c]{%
				\begin{minipage}{13cm} 
					\caption{
						Naming convention for (hom-)unital (hom-)algebras. The term \textit{(hom-)unital} refers to any of these algebras.
					}
					\label{tab:naming_convention}
					\centering
					\begin{tabular}{|c|c|c|c|}
						\hline
						\textbf{With...} & \textbf{On side...} & \textbf{Is known as...} & \textbf{ (hom-)unity...}\\
						\hline
						Unity & Left & Left-unital & $1_l$\\
						& Right & Right-unital & $1_r$\\
						& Both & Two-sided unital & $1_{lr}$\\
						& Unspecified & One-sided unital & $1_\star$\\
						\hline
						Hom-unity & Left & Left hom-unital & $\omega_l$\\
						& Right & Right hom-unital & $\omega_r$\\
						& Both & Two-sided hom-unital & $\omega_{lr}$\\
						& Unspecified & One-sided hom-unital & $\omega_\star$\\
						\hline
					\end{tabular}
				\end{minipage}
			}
		\end{table}
		
		\begin{definition}
			A (one-sided) hom-ideal of a hom-algebra $(A,\alpha)$ is a (one-sided) ideal $I\subseteq A$ such that $\alpha(I)\subseteq I$. A hom-algebra is (one-sided) hom-simple if it does not contain any nontrivial (one-sided) hom-ideals.
		\end{definition}

		An essential aspect of the structure of algebras and hom-algebras is the presence of nonzero elements that multiply to $0$.
		
		\begin{definition}
			For any subset $S\subseteq A$ of an algebra, the left annihilator of $S$ is the subspace of elements of $A$ that annihilate every element in $S$ under left multiplication,
			\begin{align}
				Ann_A^l(S)=\{a\in A\mid \forall x\in S,\ a\cdot x=0\}
				=\bigcap\limits_{x\in S}\ker(R_x).
			\end{align}
			The right annihilator of $S$ is the subspace of elements of $A$ that annihilate every element in $S$ under right multiplication,
			\begin{align}
				Ann_A^r(S)=\{a\in A\mid \forall x\in S,\ x\cdot a=0\}=\bigcap\limits_{x\in S}\ker(L_x).
			\end{align}
		\end{definition}
		An element $a\in A$ is \textit{left-regular} if $a\cdot x\neq 0 \ \forall x\in A\backslash\{0\}$, \textit{right-regular} if $x\cdot a\neq 0 \ \forall x\in A\backslash\{0\}$ and \textit{two-sided regular} if both hold. An associative algebra in which all nonzero elements are two-sided regular is an \textit{associative domain}, and one in which all nonzero elements are invertible is a \textit{division algebra}. A non-associative algebra in which all multiplication operators are injective is a \textit{non-associative domain}, and one in which all multiplication operators are bijective is a \textit{non-associative division algebra}.
		
		\section{One-sided unital hom-associative algebras}\label{sec:unital-1SU}
		From this section onward, we mainly prove results for left-unital algebras, as the right-unital counterparts are analogous, and easily obtained by considering the opposite algebra $A^{\mathrm{op}}=(A,\mu_{\mathrm{op}},1_r)$ of a given right-unital algebra $A=(A,\mu,1_r)$. The opposite product is given by $\mu_{\mathrm{op}}(x,y)=y\cdot x$. The opposite algebra $(A^{\mathrm{op}},\mu,1_r)$ is left-unital if and only if $(A,1_r)$ is right-unital. It then suffices to study the left-unital algebras.
		
		\begin{remark}
			Observe that a left unital hom-associative algebra $(A,\alpha,1_l)$ is hom-unital with hom-unity  $\alpha(1_l)$ -- indeed, by hom-associativity, we have
			\begin{equation*}
				\alpha(x)=(1_l\cdot 1_l)\cdot\alpha(x)=\alpha(1_l)\cdot(1_l\cdot x)=\alpha(1_l)\cdot x,
			\end{equation*}
			which is the hom-unitality axiom.
		\end{remark}
		
		\begin{lemma}\label{lem:unit-fundamental}
			Let $(A,\alpha,1_l)$ be a left-unital hom-associative algebra. Then, for all $x,y,z\in A$,
			\begin{multicols}{2}\raggedcolumns
				\begin{enumerate}[label=\textnormal{(\roman*)},ref=\textnormal{(\roman*)}]
					\item $\alpha(x)\cdot z=(x\cdot  1_l)\cdot \alpha(z)$,
					\item $\alpha(x)\cdot  x=(x\cdot  1_l)\cdot \alpha(x)$,
					\item $\alpha(x)\cdot 1_l=(x\cdot  1_l)\cdot \alpha(1_l)$,
					\item $\alpha(1_l)\cdot  x=\alpha(x)$,
					\item $1_l\cdot \alpha(1_l)=\alpha(1_l)\cdot  1_l$,
					\item $\alpha(x\cdot  y)=x\cdot \alpha(y)$,
					\item $\alpha=L_{\alpha(1_l)}$.
				\end{enumerate}
			\end{multicols}
			Let $(A,\alpha,1_r)$ be a right-unital hom-associative algebra. Then, for all $x,y,z\in A$,
			\begin{multicols}{2}\raggedcolumns
				\begin{enumerate}[label=\textnormal{(\roman*)},ref=\textnormal{(\roman*)}]
					\setcounter{enumi}{7}
					\item $x\cdot \alpha(z)=\alpha(x)\cdot (1_r\cdot  z)$,
					\item $x\cdot \alpha(x)=\alpha(x)\cdot (1_r\cdot  x)$,
					\item $1_r\cdot \alpha(x)=\alpha(1_r)\cdot (1_r\cdot  x)$,
					\item $x\cdot \alpha(1_r)=\alpha(x)$,
					\item $\alpha(1_r)\cdot  1_r=1_r\cdot \alpha(1_r)$,
					\item $\alpha(x\cdot  y)=\alpha(x)\cdot  y$,
					\item $\alpha=R_{\alpha(1_r)}$.
				\end{enumerate}
			\end{multicols}
		\end{lemma}
		\begin{proof}
			We apply the hom-associativity and one-sided unitality conditions:
			\begin{enumerate}[label=\textnormal{(\roman*)}$\colon$,ref=\textnormal{(\roman*)}]
				\item $(x\cdot 1_l)\cdot \alpha(z)=\alpha(x)\cdot (1_l\cdot z)=\alpha(x)\cdot z$,
				\item $(x\cdot 1_l)\cdot \alpha(x)=\alpha(x)\cdot (1_l\cdot x)=\alpha(x)\cdot x$,
				\item $(x\cdot 1_l)\cdot \alpha(1_l)=\alpha(x)\cdot (1_l\cdot 1_l)=\alpha(x)\cdot 1_l$,
				\item $(1_l\cdot 1_l)\cdot \alpha(x)=\alpha(1_l)\cdot (1_l\cdot x)=\alpha(1_l)\cdot x$,
				\item $\alpha(1_l)\cdot 1_l=\alpha(1_l)=1_l\cdot \alpha(1_l)$,
				\item $\alpha(x\cdot y)=\alpha(1_l)\cdot (x\cdot y)=(1_l\cdot x)\cdot \alpha(y)=x\cdot \alpha(y)$,
				\item $\alpha(x)=\alpha(1_l)\cdot x\ \forall x\in A \Rightarrow \alpha=L_{\alpha(1_l)}$,
				\item $\alpha(x)\cdot (1_r\cdot z)=(x\cdot 1_r)\cdot \alpha(z)=x\cdot \alpha(z)$,
				\item $\alpha(x)\cdot (1_r\cdot x)=(x\cdot 1_r)\cdot \alpha(x)=x\cdot \alpha(x)$,
				\item $\alpha(1_r)\cdot (1_r\cdot x)=(1_r\cdot 1_r)\cdot \alpha(x)=1_r\cdot \alpha(x)$,
				\item $\alpha(x)\cdot (1_r\cdot 1_r)=(x\cdot 1_r)\cdot \alpha(1_r)=x\cdot \alpha(1_r)$,
				\item $1_r\cdot \alpha(1_r)=\alpha(1_r)=\alpha(1_r)\cdot 1_r$,
				\item $\alpha(x\cdot y)=(x\cdot y)\alpha(1_r)=\alpha(x)\cdot (y\cdot 1_r)=\alpha(x)\cdot y$,
				\item $\alpha(x)=x\cdot \alpha(1_r)\ \forall x\in A \Rightarrow \alpha=R_{\alpha(1_r)}$.
			\end{enumerate}
		\end{proof}
		We represent these relations in a table for better comparison.
		\begin{table}[H]
			\makebox[\textwidth][c]{%
				\begin{minipage}{13cm} 
					\caption{Product relations in one-sided unital hom-associative algebras.}
					\label{tab:product_relations_one}
					\centering
					\begin{tabular}{|c|c|}
						\hline
						\textbf{Left-unital} &  \textbf{Right-unital}   \\
						\hline
						$\alpha(x)\cdot y=(x\cdot 1_l)\cdot \alpha(y)$ &  $x\cdot \alpha(y)=\alpha(x)\cdot (1_r\cdot y)$ \\
						$\alpha(x\cdot y)=x\cdot \alpha(y)$ &  $\alpha(x\cdot y)=\alpha(x)\cdot y$ \\
						$\alpha(x)\cdot 1_l=(x\cdot 1_l)\cdot \alpha(1_l)$ &   $1_r\cdot \alpha(x)=\alpha(1_r)\cdot (1_r\cdot x)$ \\
						$\alpha(x)=\alpha(1_l)\cdot x$ &  $\alpha(x)=x\cdot \alpha(1_r)$ \\
						$\alpha=L_{\alpha(1_l)}$ &  $\alpha=R_{\alpha(1_r)}$ \\
						$x\cdot y=1_l\Rightarrow  x\cdot \alpha(y)=\alpha(1_l)$ &  $x\cdot y=1_r\Rightarrow \alpha(x)\cdot y=1_r$ \\
						\hline
						$1_l\cdot \alpha(1_l)=\alpha(1_l)=\alpha(1_l)\cdot 1_l$ &  $\alpha(1_r)\cdot 1_r=\alpha(1_r)=1_r\cdot \alpha(1_r)$ \\
						\hline
					\end{tabular}
				\end{minipage}
			}
		\end{table}
		
		The two-sided unital analogs of the relations in \cref{tab:product_relations_one} are the opening result of \cite{FreGohr:UnitCond} and are naturally a combination of the two prior. Two-sided unitality also ensures commu\-tation between any element and its image by $\alpha$.
		
		\begin{table}[H]
			\makebox[\textwidth][c]{%
				\begin{minipage}{7.8cm} 
					\caption{Product relations in two-sided \\ unital hom-associative algebras.}
					\label{tab:product_relations_two}
					\centering
					\begin{tabular}{|c|}
						\hline
						\textbf{Two-sided unital}   \\
						\hline
						$x\cdot \alpha(y)=\alpha(x)\cdot y$ \\
						$x\cdot\alpha(y)=\alpha(x\cdot y)=\alpha(x)\cdot y$ \\
						$\alpha(1_{lr})\cdot x=\alpha(x)=x\cdot\alpha(1_{lr})$   \\
						$\alpha(1_{lr})\cdot x=\alpha(x)=x\cdot\alpha(1_{lr})$  \\
						$L_{\alpha(1_{lr})}=\alpha=R_{\alpha(1_{lr})}$  \\
						$x\cdot y=1_{lr}\Rightarrow  x\cdot\alpha(y)=\alpha(1_{lr})=\alpha(x)\cdot y$  \\
						\hline
						$x\cdot \alpha(x)=\alpha(x^2)=\alpha(x)\cdot x$  \\
						\hline
					\end{tabular}
				\end{minipage}
			}
		\end{table}
		
		The twisting map of any unital hom-associative algebra is a multiplication operator, but multiplication operators need not yield unital hom-associative algebras.
		
		\begin{example}\label{ex:counter-poly}
			Consider the ring $t\cdot \mathbb{C}[t]=\{p(t)\in \mathbb{C}[t]\mid p(0)=0\}$ of polynomials in $t$ with independent term equal to $0$ and the usual product. The multiplication operator $L_t\colon p(t)\mapsto t\cdot p(t)$ defines a hom-associative algebra, since by the multi-linearity of the hom-associator, the hom-associativity holds for any elements since it holds for the basis $\{t^n\mid n\ge 1\}$ of $t\cdot \mathbb{C}[t]$:
			\begin{multline*}
				[t^n,t^m,t^l]_{\mathrm{as}}^{L_t}=(t^n\cdot t^m)\cdot L_t(t^l)-L_t(t^n)\cdot (t^m\cdot t^l)\\=(t^n\cdot t^m)\cdot t^{l+1}-t^{n+1}\cdot (t^m\cdot t^l)=t^{n+m+l+1}-t^{n+1+m+l}=0.
			\end{multline*}
			As a hom-algebra, $(t\mathbb{C}[t],L_t)$ is hom-associative but not unital.
		\end{example}
		
		\begin{remark}
			Note that the hom-algebra in \cref{ex:counter-poly} is also associative. In the sequel (see \cref{thm:bijection_left_unital}\ref{item:HAstr_B4}), we show that hom-associativity is due to the centrality of $t$ in $\mathbb{C}[t,t^{-1}]$.
		\end{remark}
		
		\subsection{Association relations}\label{ssec:association-relations}
		
		One-sided unities need not be unique. Indeed, for any two different one-sided unities $1_l$ and $1_l'$ we have $1_l\cdot x=x=1_l'\cdot x$, and thus $(1_l'-1_l)\cdot x=0$, that is, $1_l'-1_l$ is a universal left zero divisor of $A$ or $1_l-1_l'\in Ann_A^l(A)$. Therefore, any unity in a domain is unique. The element $1_l'-1_l$ is also nilpotent with order $2$, as $(1_l'-1_l)\cdot (1_l'-1_l)=1_l'\cdot 1_l'-1_l'\cdot 1_l-1_l\cdot 1_l'+1_l\cdot 1_l=1_l'-1_l-1_l'+1_l=0$. If there are left and right unities, then the unity is two-sided and unique: for any two one-sided unities $1_l$ and $1_r$, $1_l=1_l1_r=1_r$.

		The twisting map $\alpha$ of a unital hom-associative algebra $(A,\alpha,1_\star)$ is multiplication by $\alpha(1_\star)$ (see \cref{lem:unit-fundamental}). This property applied to the hom-associators provides associativity conditions that extend those in \cite[Proposition 1.1]{FreGohr:UnitCond}.
		\begin{proposition}\label{prop:AssociatorsHomAssociators}
			Let $(A,\alpha,1_l)$ be a left-unital hom-associative algebra. Then, for all $x,y,z\in A$,
			\begin{equation*}
				[x,y,\alpha(z)]_{\mathrm{as}}=\alpha([x,y,z]_{\mathrm{as}}).
			\end{equation*}
			Moreover, if $\alpha$ is injective, then, for all $x,y,z\in A$,
			\begin{equation*}
				[x,y,\alpha(z)]_{\mathrm{as}}=0\Leftrightarrow [x,y,z]_{\mathrm{as}}=0, \textnormal{ that is, }\alpha(z)\in N^r(A)\Leftrightarrow z\in N^r(A)
			\end{equation*}
			Let $(A,\alpha,1_r)$ be a right-unital hom-associative algebra. Then, for all $x,y,z\in A$,
			\begin{equation*}
				[\alpha(x),y,z]_{\mathrm{as}}=\alpha([x,y,z]_{\mathrm{as}}).
			\end{equation*}
			Moreover, if $\alpha$ is injective, then, for all $x,y,z\in A$,
			\begin{equation*}
				[\alpha(x),y,z]_{\mathrm{as}}=0\Leftrightarrow [x,y,z]_{\mathrm{as}}=0, \textnormal{ that is, }\alpha(x)\in N^l(A)\Leftrightarrow x\in N^l(A)
			\end{equation*}
		\end{proposition}
		\begin{proof}
			In left-unital hom-associative algebras, it follows from \cref{lem:unit-fundamental} that for all $x,y\in A$,  $\alpha(x\cdot y)=x\cdot \alpha(y)$. Applying this property to the elements $\{x\cdot y,z\}$ yields
			\begin{multline*}
				(x\cdot y)\cdot \alpha(z)=\alpha((x\cdot y)\cdot z)\Rightarrow [x,y,\alpha(z)]_{\mathrm{as}}=(x\cdot y)\cdot \alpha(z)-x\cdot (y\cdot \alpha(z))\\
				=\alpha((x\cdot y)\cdot z)-x\cdot \alpha(y\cdot z)
				=\alpha((x\cdot y)\cdot z)-\alpha(x\cdot (y\cdot z))=
				\alpha([x,y,z]_{\mathrm{as}}).
			\end{multline*}
			If $\alpha$ is injective, then $[x,y,\alpha(z)]_{\mathrm{as}}=\alpha([x,y,z]_{\mathrm{as}})=0\Leftrightarrow [x,y,z]_{\mathrm{as}}=0$.
			
			In right-unital hom-associative algebras, it follows from \cref{lem:unit-fundamental} that for all $x,y\in A$, $\alpha(x\cdot y)=\alpha(x)\cdot y$. Applying this property to elements $\{x,y\cdot z\}$ yields
			\begin{multline*}
				\alpha(x)\cdot (y\cdot z)=\alpha(x\cdot (y\cdot z))\Rightarrow [\alpha(x),y,z]_{\mathrm{as}}=(\alpha(x)\cdot y)\cdot z-\alpha(x)\cdot (y\cdot z)\\
				=\alpha(x\cdot y)\cdot z-\alpha(x\cdot (y\cdot z))
				=\alpha((x\cdot y)\cdot z)-\alpha(x\cdot (y\cdot z))=
				\alpha([x,y,z]_{\mathrm{as}})
			\end{multline*}
			If $\alpha$ is injective, then $[\alpha(x),y,z]_{\mathrm{as}}=\alpha([x,y,z]_{\mathrm{as}})=0\Leftrightarrow [x,y,z]_{\mathrm{as}}=0$.
		\end{proof}
		In the two-sided unital case, both statements in \cref{prop:AssociatorsHomAssociators} hold. It follows that $\alpha((x\cdot y)\cdot z)=\alpha(x\cdot (y\cdot z))$ for all $x,y,z\in A$, that is, $\alpha([x,y,z]_{\mathrm{as}})=0$. Applying this property to \cref{prop:AssociatorsHomAssociators} gives \cite[Proposition 1.1]{FreGohr:UnitCond}, which can be used to prove the following identity:
		\begin{equation*}
			[\alpha(x),y,z]_{\mathrm{as}}=[x,\alpha(y),z]_{\mathrm{as}}=[x,y,\alpha(z)]_{\mathrm{as}}=\alpha([x,y,z]_{\mathrm{as}})=0.
		\end{equation*}
		Thus, $\alpha(A)\subseteq N(A)$.
		This identity applied to a surjective $\alpha$ implies that $(A,1_\star)$ is associative. This is another proof of  \cite[Corollary 1.1]{FreGohr:UnitCond}.
		
		\subsection{Multiplicativity and associativity conditions}\label{ssec:multiplicativity}
		In the literature, hom-algebras are often considered multiplicative \cite{AmmEjbMak:cohomology,MakZah:HAclassification,ArfFraMak:morphisms,Mak:HomAltHomJordan} or constructed multiplicative by twisting the product of an algebra by an algebra homomorphism \cite{BackRic:HAWeyl,BackRicSil:HAOreWeakUnit,Yau:HAhomology}. This allows a direct generalization of many relevant algebraic properties, including cohomology \cite{AmmEjbMak:cohomology} and deformations \cite{ArfFraMak:morphisms}, and multiplicative hom-associative algebras are classified in low dimensions \cite{MakZah:HAclassification}.
		Multiplicativity in one-sided unital algebras behaves as follows.
		\begin{theorem}\label{thm:unital_multiplicativity}
			Let $(A,\alpha,1_\star)$ be a unital hom-associative algebra. The following statements are equivalent.
			\begin{multicols}{2}
				\begin{enumerate}[label=\textnormal{\arabic*)},ref=\textnormal{\arabic*}]
					\item\label{eq:mult_cond_1} $\alpha$ is multiplicative,
					\item\label{eq:mult_cond_2} $\alpha$ is idempotent,
					\item\label{eq:mult_cond_3} $\alpha^2(1_\star)=\alpha(1_\star)$,
					\item\label{eq:mult_cond_4} $\alpha(1_\star)$ is idempotent.
				\end{enumerate}
			\end{multicols}
		\end{theorem}
		\begin{proof} We prove \ref{eq:mult_cond_1} $\Rightarrow$ \ref{eq:mult_cond_2} $\Rightarrow$ \ref{eq:mult_cond_3} $\Rightarrow$ \ref{eq:mult_cond_4} $\Rightarrow$ \ref{eq:mult_cond_1} for a left-unital algebra $(A,1_l)$:
			\begin{enumerate}[label=,leftmargin=*]
				\item\ref{eq:mult_cond_1}$\Rightarrow$\ref{eq:mult_cond_2}: If $\alpha$ is multiplicative, then $x\cdot \alpha(y)=\alpha(x\cdot y)=\alpha(x)\cdot \alpha(y)$. It follows that $x\cdot \alpha(y)=\alpha(x)\cdot \alpha(y)$, and by contracting $x\mapsto 1_l$ we obtain $ \alpha(y)=\alpha(1_l)\cdot \alpha(y)=\alpha^2(y)$ and thus $\alpha^2=\alpha$.  That is, $\alpha$ is an idempotent.
				\item\ref{eq:mult_cond_2}$\Rightarrow$\ref{eq:mult_cond_3}: If $\alpha^2=\alpha$, then $\alpha(1_l)^2=\alpha(1_l)\cdot \alpha(1_l)=\alpha^2(1_l)=\alpha(1_l)$. It follows that $\alpha^2(1_l)=\alpha(1_l)^2=\alpha(1_l)$. That is, $\alpha(1_l)$ is idempotent.
				\item\ref{eq:mult_cond_3}$\Rightarrow$\ref{eq:mult_cond_4}: From $\alpha(1_l)^2=\alpha^2(1_l)$ it follows that $\alpha(1_l)^2=\alpha(1_l)\Leftrightarrow \alpha^2(1_l)=\alpha(1_l)$.
				\item\ref{eq:mult_cond_4}$\Rightarrow$\ref{eq:mult_cond_1}: If $\alpha^2(1_l)=\alpha(1_l)$, then $\alpha(x)\cdot \alpha(y)=(\alpha(1_l)\cdot x)\cdot \alpha(y)$. By hom-associati\-vity, this is equal to $\alpha^2(1_l)\cdot (x\cdot y)=\alpha(1_l)\cdot (x\cdot y)=\alpha(x\cdot y)$. It follows that $\alpha(x\cdot y)=\alpha(x)\cdot \alpha(y)$. That is, $\alpha$ is multiplicative.
			\end{enumerate}
			The right-unital case is analogous to the left-unital case. This completes the proof.
		\end{proof}
		We can visualize these relations using the following diagram:
		\begin{center}
			\begin{tikzcd}[arrows=Rightarrow]	
				\alpha \textnormal{ multiplicative}
				\arrow{d}
				&
				\alpha^2(1_\star)=\alpha(1_\star)
				\arrow{l}
				\arrow[Leftrightarrow]{d} \\
				
				\alpha \textnormal{ is idempotent}
				\arrow{r}
				&
				\alpha(1_\star) \textnormal{ idempotent}
			\end{tikzcd}
		\end{center}
		
		The twisting map of any unital hom-associative algebra is a multiplication operator, and thus is determined (albeit not always uniquely) by the factor $\alpha(1_\star)$. In other words, there is a correspondence between multiplicative twisting maps and a certain subset of idempotents of the algebra (see \cref{thm:bijection_left_unital}).
		
		\begin{remark}
			The multiplicativity of $\alpha$ can be determined by computing $\alpha(1_\star)^2$. This allows a very streamlined classification of multiplicative hom-associative structures, as will be seen in the sequel, and can moreover be used to rule out hom-associative structures on certain algebras.
		\end{remark}

		\begin{corollary}
			The following equivalent statements are true:
			\begin{enumerate}[label=\textnormal{\arabic*)}]
				\item No multiplicative hom-associative algebra $(A,\alpha)$ with $\alpha^2\neq\alpha$ is unital.
				\item A unital multiplicative hom-associative algebra without non-trivial idempotents is either trivial or associative.
			\end{enumerate}
		\end{corollary}
		
		\begin{example}
			The $3$-dimensional multiplicative hom-associative algebras $A_4^3$, $A_5^3$,  $A_6^3$ and $A_{11}^3$, classified in \textnormal{\cite{MakZah:HAclassification}}, have non-idempotent twisting maps {\rm (}see \textnormal{\cite{GarButSil:ZeroDivision}}{\rm )}, and thus are non-unital.
		\end{example}
		
		\begin{example}
			Any division algebra has no non-trivial idempotents, and thus it cannot be seen as multiplicative hom-associative unless $\alpha(1_\star)=1_\star$, in which case, it is associative.
		\end{example}
		
		\begin{example}[\!\!\cite{SehZas:GRidempotents}]
			Let $R$ be a unital ring without non-trivial idempotents, $G$ a finite group such that every prime divisor of the order of $G$ is not invertible in $R$. The group ring $RG$ has no nontrivial idempotents. As a unital $R$-algebra, $(RG,1e_G)$ cannot be seen as multiplicative hom-associative unless $\alpha(e_G)=e_G$, in which case it is associative.
		\end{example}
		
		\begin{corollary}
			Let $(A,\alpha,1_\star)$ be a unital hom-associative algebra. If any of the statements in \cref{thm:unital_multiplicativity} holds and $\alpha$ is injective or $1_\star\in \alpha(A)$ (and in particular, if $\alpha$ is surjective), then $(A,1_l)$ is associative.
		\end{corollary}
		\begin{proof}
			Let $(A,\alpha,1_\star)$ be a one-sided unital hom-associative algebra. If any of the statements in  \cref{thm:unital_multiplicativity} holds, then $\alpha$ is idempotent. Moreover,
			
			\begin{enumerate}
				\item \textbf{If $\alpha$ is injective:} From  $\alpha$ idempotent folows that $\alpha(\alpha(x)-x)=0$. By the injectivity of $\alpha$,  $\alpha(x)=x$, and thus $\alpha=\mathrm{id}_ A $.
				\item \textbf{If $1_\star\in\alpha(A)$:} there is $a\in A$ such that $1_\star=\alpha(a)=\alpha^2(a)=\alpha(1_\star)$. It follows that $\alpha=L_{1_\star}=\mathrm{id}_A$ or $\alpha=R_{1_\star}=\mathrm{id}_A$.
				\item \textbf{If $\alpha$ surjective:} $1_\star\in\alpha(A)$ always holds, as $\alpha(A)=A$.
			\end{enumerate}
			By the twisting map being the identity, $(A,1_\star)$ is associative.
		\end{proof}
		This result expands \cite[Corollary 1.3]{FreGohr:UnitCond} to one-sided unital hom-associative algebras and illustrates how restrictive the multiplicativity of $\alpha$ is: multiplicative unital hom-associative algebras are often just associative, in which case it is much more interesting to study them with the machinery of associative algebras.

		\section{Hom-unities in two-sided unital algebras}\label{sec:hom-unities-2SU}
		
		\begin{definition}
			Given an algebra $A$, a linear map $\alpha\colon A\to A$ is \emph{HA-compatible}\index{HA-compatible}\index{compatible!hom-associative} or a \emph{hom-associative structure of $A$} if $(A,\alpha)$ is hom-associative.
			
			An element $a\in A$ \emph{induces a hom-associative structure} on $A$ if the multiplication operator $\alpha=L_a$ or $\alpha=R_a$ is HA-compatible.
			
			If no nonzero map is HA-compatible, we say that $A$ has no nontrivial hom-associative structures.
		\end{definition}
		
		The twisting map $\alpha$ of a left-unital hom-associative algebra is a multiplication operator $L_a$, and thus $a$ satisfies the hom-unitality axiom $\alpha(x)=a\cdot x$. The search for hom-associative structures is then one with the search for hom-unities of the algebra that induce hom-associative structures.
		
		In two-sided unital algebras, the subspace of HA-compatible maps consists of left multiplications by elements in the following subalgebra (See \cite[Proposition 1.1]{FreGohr:UnitCond})
		\begin{equation*}
			AC(A)=\{a\in Z(A)\mid Aa\subseteq N(A) \textnormal{ is a two-sided ideal}\}.
		\end{equation*}
		Any $a\in AC(A)$ is central. Thus, considering the left or right multiplication is a matter of preference. $AC(A)$ is the subalgebra consisting of all hom-unities of $(A,1_{lr})$ that induce hom-associative structures.
		\begin{theorem}\label{thm:AC_two_sided_unital}
			Let $(A,1_{lr})$ be a two-sided unital algebra. With the notations above,
			\begin{equation}\label{eqn:ACA_formula}
				AC(A)=Z(A)\cap N(A)\cap Ann_A^l([A,A,A]_{\mathrm{as}}).
			\end{equation}
		\end{theorem}
		\begin{proof}
			For any $a\in AC(A)$, it follows from $Aa\subseteq N(A)$ that $a\in N(A)$. Thus, $AC(A)\subseteq Z(A)\cap N(A)$. Moreover, $Aa\subseteq N(A)$ is a two-sided ideal only if the following association relations hold for all $x,y,z\in A$:
			\begin{align*}
				[ax,y,z]_{\mathrm{as}}&=((ax)y)z-(ax)(yz)=(a(xy))z-a(x(yz))\\
				&\hspace{4.25cm}=a((xy)z)-a(x(yz))=a[x,y,z]_{\mathrm{as}},\\
				[x,ay,z]_{\mathrm{as}}&=(x(ay))z-x((ay)z)=((xa)y)z-x(a(yz))\\
				&\hspace{1.25cm}=((ax)y)z-(ax)(yz)=(a(xy))z-a(x(yz))\\
				&\hspace{4.25cm}=a((xy)z)-a(x(yz))=a[x,y,z]_{\mathrm{as}},\\
				[x,y,za]_{\mathrm{as}}&=(xy)(za)-x(y(za))=((xy)z)a-x((yz)a)\\
				&\hspace{2cm}=((xy)z)a-(x(yz))a=[x,y,z]_{\mathrm{as}}a=a[x,y,z]_{\mathrm{as}}.
			\end{align*}
			These relations indicate that $Aa\subseteq N(A)$ is a two-sided ideal if and only if $a[x,y,z]_{\mathrm{as}}=0$ for all $x,y,z\in A$, that is, $a\in Ann_{A}^l([A,A,A]_{\mathrm{as}})$. Finally, ideal conditions for $Aa$ follow from $a\in Z(A)\cap N(A)$:
			\begin{align*}
				&\textbf{Left ideal: }\text{for all }  b\in A,\ b(ax)=b(xa)=(bx)a\in Aa. \\
				&\textbf{Right ideal: }\text{for all } b\in A,\ (xa)b=x(ab)=x(ba)=(xb)a\in Aa.
			\end{align*}
			This completes the proof.
		\end{proof}
		Because this subalgebra consists entirely of central elements, it can be rewritten in any of the following forms:
		\begin{align*}
			AC(A)=Z(A)\cap N(A)\cap Ann_A^l([A,A,A]_{\mathrm{as}})& \\
			=Z(A)\cap N(A)\cap Ann_A^r([A,A,A]_{\mathrm{as}})&\\
			=Z(A)\cap N(A)\cap Ann_A([A,A,A]_{\mathrm{as}})&.
		\end{align*}
		
		\begin{example}\label{ex:two_sided_unital_associative}
			Any two-sided unital associative algebra $(A,1_{lr})$ satisfies $N(A)=A$ and $[A,A,A]_{\mathrm{as}}=\{0\}$, and thus
			$AC(A)=Z(A)$. That is, (the product in) $(A,\alpha,1_{lr})$ is hom-associative if and only if the twisting map is $\alpha=L_\omega$ for some $\omega\in Z(A)$. Then, $(A,L_\omega,\omega)$ is two-sided hom-unital hom-associative. The algebra in Example \ref{ex:counter-poly} is a case of this, with hom-unity $\omega=t$.
		\end{example}

		\begin{example}\label{ex:two_sided_unital_simple}
			Any two-sided unital simple algebra with nontrivial HA-compatible maps is associative. Indeed, if $(A,1_{lr})$ is two-sided unital simple and $AC(A)\neq\{0\}$, then $Aa$ is a two-sided ideal for some $a\neq 0$ in $AC(A)$. This implies one of the following:
			\begin{align*}
				A a&=\{0\}\overset{a\in Z(A)}{\Longrightarrow} L_a=0\Rightarrow Twist(A)=\{0\}\Rightarrow a=0,\\
				A a&= A \overset{ A a\subseteq N(A)}{\Longrightarrow} N(A)= A \Rightarrow (A,1_{lr}) \textnormal{ is associative.}
			\end{align*}
			By associativity, as seen in Example \ref{ex:two_sided_unital_associative}, we have $AC(A)=Z(A)$.
		\end{example}
		
		\begin{corollary}\label{cor:HA_non_injective}
			No HA-compatible map of a two-sided unital nonassociative algebra is injective. Moreover, if there exist $x,y,z\in A$ such that $[x,y,z]_{\mathrm{as}}$ is one-sided regular, then $AC(A)=\{0\}$.
		\end{corollary}
		\begin{proof}
			Let $a\in AC(A)$. Then $a\in Ann_A^l([A,A,A]_{\mathrm{as}})$. We can write this in the form $a\cdot [A,A,A]_{\mathrm{as}}=\{0\}$ or $[A,A,A]_{\mathrm{as}}\subseteq\ker L_a$. Since $A$ is non-associative,  $\ker L_a$ is nontrivial. If, for some $x,y,z\in A$,  $[x,y,z]_{\mathrm{as}}$ is (without loss of generality) right-regular, then there is no nonzero $a\in A$ such that $[x,y,z]_{\mathrm{as}}\in \ker L_a$, that is, $Ann_A^l([A,A,A]_{\mathrm{as}})=\{0\}$. Conclusion follows.
		\end{proof}

		All hom-associative structures of a given two-sided unital algebra are given by multiplication with central elements; that is, the twisting maps are determined by two-sided hom-unities. This problem can be reversed in different ways, which leaves two interesting open questions.
		\begin{enumerate}
			\item Can a given hom-unital hom-associative algebra be interpreted as a hom-assoc\-iative structure over another algebra?
			\item Given a linear map of $A$, when is it possible to find a product in $A$ that makes it HA-compatible?
		\end{enumerate}
		
		\begin{example}
			The first Cayley-Dickson algebras over $\mathbb{R}$ are the complex numbers $\mathbb{C}$, the quaternions $\mathbb{H}$, the octonions $\mathbb{O}$, the sedenions $\mathbb{S}$ and more, each obtained by \textit{doubling} the previous one, by adding a new generator and a new involution.
			\begin{equation*}
				\begin{array}{c|c|c|c}
					\textnormal{Algebra} & \textnormal{Construction} & \textnormal{Generators} & \textnormal{Properties}\\
					\hline
					\mathbb{C} & \mathbb{R}\oplus\mathbb{R}i & \{1,i\} & i^2=-1.\\
					\mathbb{H} & \mathbb{C}\oplus\mathbb{C}j & \{1,i,j,k\} & ij=k,\ ijk=-1.
				\end{array}
			\end{equation*}
			Using this process twice more we obtain $\mathbb{O}$ with 8 generators and $\mathbb{S}$ with 16.
			
			Each step of the Cayley-Dickson process doubles the dimension in exchange for certain algebraic properties: $\mathbb{R}$ is commutative associative domain, $\mathbb{H}$ is non-commutative associative domain, $\mathbb{O}$ is non-commutative alternative domain, $\mathbb{S}$ is non-commutative flexible and has zero divisors. All of these algebras are two-sided unital with unity $1$. The coefficient in $1$ is called the \emph{real part} of these objects, and we identify the subalgebra of \emph{real} elements (with only real part) with $\mathbb{R}$.
			The hom-unities inducing HA-compatible maps in $\mathbb{C}$, $\mathbb{H}$, $\mathbb{O}$ and $\mathbb{S}$ are
			\begin{enumerate}[label=\textnormal{\arabic*)}]
				\item $\mathbb{C}\colon\ AC(\mathbb{C})=\mathbb{C}$.
				\item $\mathbb{H}\colon\ AC(\mathbb{H})=Re(\mathbb{H})=\mathbb{R}\cdot 1$.
				\item $\mathbb{O}\colon\ AC(\mathbb{O})=\{0\}$.
				\item $\mathbb{S}:\ AC(\mathbb{S})=\{0\}$.
			\end{enumerate}
			
			We use the explicit formulation given by \cref{thm:AC_two_sided_unital}, namely
			\begin{equation*}
				AC(A)=Z(A)\cap N(A)\cap Ann_A^l([A,A,A]_{\mathrm{as}}),
			\end{equation*}
			given that $\mathbb{C}, \mathbb{H},\mathbb{O}$ and $\mathbb{S}$ are two-sided unital. We can then see all HA-compatible maps as multiplication operators, following \cref{tab:product_relations_one}.
			
			$\mathbb{C}$ is a commutative associative domain; thus, $[\mathbb{C},\mathbb{C},\mathbb{C}]_{\mathrm{as}}=\{0\}$. It follows that $AC(\mathbb{C})=Z(\mathbb{C})\cap N(\mathbb{C})\cap Ann_\mathbb{C}^l(\{0\})=\mathbb{C}$.
			
			$\mathbb{H}$ is an associative domain; thus, $[\mathbb{H},\mathbb{H},\mathbb{H}]_{\mathrm{as}}=\{0\}$. It follows that $$AC(\mathbb{H})=Z(\mathbb{H})\cap N(\mathbb{H})\cap Ann_\mathbb{H}^l(\{0\})=Re(\mathbb{H})\cap \mathbb{H}\cap \mathbb{H}=Re(\mathbb{H})=\mathbb{R}1.$$
			
			$\mathbb{O}$ is a domain. We have $Ann_\mathbb{O}^l([\mathbb{O},\mathbb{O},\mathbb{O}]_{\mathrm{as}})=\{0\}$. It follows that $AC(\mathbb{O})=Z(\mathbb{O})\cap N(\mathbb{O})\cap\{0\}=\{0\}$.
			
			Regarding $\mathbb{S}$, observe that $Z(\mathbb{S})=N(\mathbb{S})=Re(\mathbb{S})\cong\mathbb{R}$. From the definition of $AC(\mathbb{S})$ it follows that $AC(A)\cong T\subseteq\mathbb{R}$.  Since $\mathbb{S}$ is non-associative, there exists a nonzero element  $z\in [\mathbb{S},\mathbb{S},\mathbb{S}]_{\mathrm{as}}$. The real part $e_0$ of any sedenion associate, $N(\mathbb{S})=Re(\mathbb{S})$, and thus $z$ has the form $z=\sum\limits_{i=1}^{15}a_ie_i$. For all $\lambda\in\mathbb{R}$, it holds that $\lambda z=0\Leftrightarrow \lambda a_i=0$ for all $i\in\{1,\dots,15\}$. Given that $z\neq 0$, there is at least one $a_i\neq 0$, and thus $\lambda=0$. It follows that $Ann_\mathbb{S}^l([\mathbb{S},\mathbb{S},\mathbb{S}]_{\mathrm{as}})\cap Re(\mathbb{S})=\{0\}$, from which $AC(\mathbb{S})=\{0\}$.
			\begin{remark}
				$AC(A)$ is a subalgebra of any two-sided unital algebra $A$. Then, $0$ is the hom-unity corresponding to the zero map, which we can see as $0\mathrm{id}_A$ or multiplication by $0\cdot 1_{lr}$.
			\end{remark}
			It is shown that all left multiplications in $\mathbb{C}$ are HA-compatible, while $\mathbb{H}$, $\mathbb{O}$ and $\mathbb{S}$ admit no nontrivial hom-associative structures. Every HA-compatible map is either a real multiple of the identity (in the case of $\mathbb{H}$) or the zero map, seen as multiplication by $0$.
			The Cayley-Dickson process removes association relations: from the associative algebra $\mathbb{H}$ we obtain the alternative algebra $\mathbb{O}$ and then the flexible algebra $\mathbb{S}$. The loss of hom-unities from $\mathbb{H}$ to $\mathbb{O}$ is due to the loss of associativity while maintaining the domain property.
		\end{example}

		\section{Hom-unities in one-sided unital algebras}\label{sec:hom-unities-1SU}
		In this section, we study left-unital algebras, which we denote $(A,1_l)$. The right-unital versions of all results in this section follow by applying them to the opposite algebra $A^{\mathrm{op}}=(A,\mu_{\mathrm{op}},1_\star)$ -- as seen at the beginning of \cref{sec:unital-1SU}, $A^{\mathrm{op}}$ which is left-unital (right-unital) if and only if $(A,\mu,1_\star)$ is right-unital (left-unital) with the same unity.
		
		Consider the maps
		\begin{align}
			\begin{array}{ccc}
				\phi:Twist(A)& \rightarrow &   A   \\
				\beta & \mapsto & \beta(1_l)
			\end{array} \quad \text{ and } \quad
			\begin{array}{ccc}
				\psi:AC(A)& \rightarrow &  \mathcal{L}(A)   \\
				a & \mapsto & L_a:x\mapsto ax
			\end{array}
		\end{align}
		If $(A,1_\star)$ is two-sided unital, Fr\'{e}gier and Gohr~\cite{FreGohr:UnitCond} proved that these maps induce a bijection between all HA-compatible maps of $A$ and the subalgebra $AC(A)$ of hom-unities of the given product.
		
		The aim of this section is to describe the linear subspace of left hom-unities of a given left-unital algebra.
		
		The set of all left hom-unities with respect to all linear maps on any algebra form a linear subspace of $A$ since for any linear maps $\alpha$ and $\beta$ on $A$ and for any $x\in A$ and $\lambda,\mu\in\mathbb{F}$,
		$$ax=\alpha(x),\ bx=\beta(x) \Rightarrow
		\begin{array}{l}
			(\lambda a +\mu b) x= \lambda (ax)+\mu (bx) \\
			\quad = \lambda \alpha(x)+\mu\beta(x)=(\lambda \alpha+\mu\beta)(x),
		\end{array}$$
		and hence, in particular, if $a$ is a hom-unity with respect to $\alpha$ and $b$ is a hom-unity with respect to $\beta$, then $\lambda a +\mu b$ is a hom-unity with respect to the linear map $\lambda \alpha+\mu\beta$. In other words, if  $\alpha=L_a$ and $\beta=L_b$ for some $a,b\in A$, then
		$\lambda \alpha+\mu\beta=L_{\lambda a +\mu b}$ which means that for any algebra $A$, the set of all linear operators of multiplication from the left by elements of the algebra forms a linear subspace of the space of all linear operators on the algebra $A$. In most algebras, it is a proper subspace because the elements of this subspace are special linear transformations given by the multiplication operation of the algebra. In other words, in most algebras, the linear subspace of all left hom-unities with respect to all linear maps forms a proper linear subspace of $A$.
		For given elements $x,y,z\in A$ in an algebra $A$, the hom-associator
		$[x,y,z]_{\mathrm{as}}^\alpha=(x\cdot y)\cdot \alpha(z)-\alpha(x)\cdot (y\cdot z)$ is a linear map $\alpha\in {\cal L}(A) \to [x,y,z]_{\mathrm{as}}^\alpha \in A$ from the linear space of all linear transformations ${\cal L}(A)$ of $A$  to $A$.
		The kernel of this linear map is a linear subspace of ${\cal L}(A)$ of all linear maps $\alpha$ such that the elements $x,y,z$ hom-associate with respect to twisting map $\alpha$, that is
		$\alpha(x)\cdot (y\cdot z)=(x\cdot y)\cdot \alpha(z)$. The intersection of these associator kernel subspaces is the linear subspace of ${\cal L}(A)$ consisting of all hom-associative structures on $A$, that is, of all linear maps $\alpha$ for which the hom-algebra $(A,\alpha)$ is a hom-associative (such $\alpha$ for which $\alpha(x)(yz)=(xy)\alpha(z)$ for all $x,y,z\in A$).
		To make this more explicit, if $(A,\cdot,\alpha)$ and if $(A,\cdot,\beta)$ are hom-associative, then for any $\lambda,\mu\in\mathbb{F}$,
		\begin{multline*}
			(\lambda \alpha+\mu \beta)(x)\cdot (y\cdot z)=\lambda \alpha(x)\cdot (y\cdot z)+\mu \beta(x)\cdot (y\cdot z)\\
			=(x\cdot y)\cdot \lambda \alpha(z)+(x\cdot y)\cdot \mu \beta(z)=(x\cdot y)\cdot (\lambda \alpha+\mu \beta)(z).
		\end{multline*}
		Thus, $\lambda \alpha+\mu \beta$ is a hom-associative structure on $A$.
		The intersection of the linear subspace in ${\cal L}(A)$ of hom-associative structures on an algebra $A$ with the linear subspace of linear maps of multiplication by elements of the algebra from the left is a linear subspace of ${\cal L}(A)$ consisting of the left multiplication operators $L_a, a\in A$ such that $(A,L_a)$ is hom-associative.
		In other words, $$\{L_a \mid a \text{ is left hom-unity of } A,\ (A,L_a) \text{ is hom-associative}\}$$ is a linear subspace of ${\cal L}(A)$. Since the map $a\mapsto L_a$ from $A$ to ${\cal L}(A)$ is linear, $$\{a\in A \mid a \text{ is hom-unity of } A,\ (A,L_a) \text{ is hom-associative}\},$$
		is a linear subspace of $A$ as the preimage of
		the linear subspace $$\{L_a \mid a \text{ is hom-unity of } A,\ (A,L_a) \text{ is hom-associative}\}$$ of ${\cal L}(A)$ under this linear map.
		
		Any hom-unity $a$ of a left-unital hom-associative algebra $(A,L_a,1_l)$ satisfies the following three properties for all $x,y,z\in A$:
		\begin{enumerate}[label=\textnormal{\arabic*)}]
			\item $L_{ax}\circ L_y=L_{xy}\circ L_a$, or $(a\cdot x)\cdot (y\cdot z)=(x\cdot y)\cdot (a\cdot z)$.
			\item $L_a\circ L_x=L_x\circ L_a$, or $a\cdot (x\cdot z)=x\cdot (a\cdot z)$.
			\item $L_{ax}\circ L_y=L_a\circ L_{xy}$, or $(a\cdot x)\cdot (y\cdot z)=a\cdot ((x\cdot y)\cdot z)$.
		\end{enumerate}
		The first property is hom-associativity; the second property results from applying hom-associativity to a triple of the form $\{1_l,x,z\}$, as a direct consequence of left unitality (see \cref{tab:product_relations_one}); the third property follows from the other two.
		
		Consider the linear subspace of elements of $A$ satisfying, for all $x,y,z\in A$, the last two conditions, that is,
		\begin{align*}
			&AC_l(A)=\{a\in A \mid L_a\circ L_x=L_x\circ L_a,\ L_{ax}\circ L_y=L_a\circ L_{xy},\ \forall x,y\in A\}\\
			&=\{a\in A \mid \underset{-\textnormal{Property \ref{item:ACL1}} -}{a\cdot (x\cdot z)=x\cdot (a\cdot z)},\ \underset{-\textnormal{Property \ref{item:ACL2} } -}{(a\cdot x)\cdot (y\cdot z)=a\cdot ((x\cdot y)\cdot z)},\ \forall x,y,z\in A\}
		\end{align*}
		\begin{remark}
			Elements in $AC_l(A)$ satisfy multiple properties, which are explored in \cref{ssec:relations-NA-ACL}. \textnormal{\ref{item:ACL1}} and \textnormal{\ref{item:ACL2}} appear in \cref{tab:ACL}, also listed as \textnormal{\ref{item:ACL1}} and \textnormal{\ref{item:ACL2}}. This notation is used throughout the remainder of this paper.
		\end{remark}
		Given one-sided unitality, we can define the subspace
		\begin{equation}\label{eqn:ACL1_definition}
			AC_l(A)\cdot 1_l=\{a\cdot 1_l\mid a\in AC_l(A)\},
		\end{equation}
		which interacts with $AC_l(A)$ as follows:
		\begin{lemma}\label{lem:ACLcommutes}
			Let $(A,1_l)$ be a left-unital algebra. With the notations above, it holds that
			\begin{equation*}
				\{a\in AC_l(A)\mid a\cdot 1_l=a\}=AC_l(A)\cdot 1_l\subseteq AC_l(A).
			\end{equation*}
		\end{lemma}
		
		\begin{proof}
			Let $b\in AC_l(A)\cdot 1_l$. There exists $a\in AC_l(A)$ such that $b=a\cdot 1_l$. We check properties \ref{item:ACL1} and \ref{item:ACL2} for $b$ using those of $a$:
			\begin{align*}
				\ref{item:ACL1}&\colon L_{a1_l}=L_a \overset{\ref{item:ACL1}_a}{\Rightarrow} L_{a1_l}\circ L_x=L_x\circ L_{a1_l}\\
				\ref{item:ACL2}&\colon (a\cdot 1_l)\cdot ((x\cdot y)\cdot z)=a\cdot ((1_l\cdot (x\cdot y))\cdot z)=a\cdot ((x\cdot y)\cdot z)=(a\cdot x)\cdot (y\cdot z)\\
				&\hspace{9.25cm}\overset{\ref{item:ACL2}_a}{=}((a\cdot 1_l)\cdot x)\cdot (y\cdot z).
			\end{align*}
			Finally, by contracting $x\mapsto 1_l$ in \ref{item:ACL2}, we have $(a\cdot 1_l)\cdot (y\cdot z)=a\cdot ((1_l\cdot y)\cdot z)=a\cdot (y\cdot z)$. Now contract $y\mapsto 1_l$, we obtain $(a\cdot 1_l)\cdot z=a\cdot (1_l\cdot z)=a\cdot z$. It follows that $L_a=L_{a1_l}$.
			
			To prove the equality: For every $b\in AC_l(A)\cdot 1_l$ we have $b\cdot 1_l=b$, and for every $a\in AC_l(A)$ commuting with $1_l$ we have $a=a\cdot 1_l\Rightarrow a\in AC_l(A)\cdot 1_l$.
		\end{proof}
		
		The equality $AC_l(A)\cdot 1_l=AC_l(A)$ holds when $Ann_ A^l(A)=\{0\}$. This is, in fact, an equivalence, which we prove in \cref{thm:ACL_split}.
		
		$AC_l(A)\cdot 1_l$ is not straightforward to compute, and there is no decomposition in the center, nucleus, and annihilators, as in two-sided unital algebras \eqref{eqn:ACA_formula}. Nonetheless,  its subspace admits such a decomposition.
		
		\begin{proposition}\label{prop:ACL1_subspace}
			Let $(A,1_l)$ be a left-unital algebra. Consider the subspace
			\begin{equation*}
				HU_n^l(A)=Z(A\cdot A)\cap N^l(A)\cap N^m(A)\cap Ann_A^l([A,A,A]_{\mathrm{as}}).
			\end{equation*}
			It holds that $HU_n^l(A)\subseteq AC_l(A)\cdot 1_l$.
		\end{proposition}
		\begin{proof}
			Let $a\in HU_n^l(A)$. We can express the hom-associator in $(A,L_a)$ as
			\begin{align*}
				&(a\cdot x)\cdot (y\cdot z)=a\cdot (x\cdot (y\cdot z)),\\
				&(x\cdot y)\cdot (a\cdot z)=((x\cdot y)\cdot a)\cdot z=(a\cdot (x\cdot y))\cdot z=a\cdot ((x\cdot y)\cdot z) \\
				&\hspace{1.8cm}\Rightarrow [x,y,z]_{\mathrm{as}}^{L_a}=a\cdot ((x\cdot y)\cdot z)-a\cdot (x\cdot (y\cdot z))=a\cdot [x,y,z]_{\mathrm{as}}=0.
			\end{align*}
			It follows that $a\in AC_l(A)$. Commutation with $1_l$ follows from $a\in Z(A\cdot A)=Z(A)$, since $A$ is left-unital. This completes the proof.
		\end{proof}
		\begin{remark}
			Since $(A,1_l)$ is left-unital, $Z(A\cdot A)\cap N^l(A)\cap N^m(A)=Z(A)\cap N(A)$, and thus the subspace we call $HU_n^l(A)$ is no other than $AC(A)$ as defined in the previous section, all hom-unities in $HU_n^l(A)$ are then two-sided. The equality $HU_n^l(A)=AC(A)$ is broken when exploring non-unital algebras.
		\end{remark}
		
		The equality $AC_l(A)\cdot 1_l=HU_n^l(A)$ is not given. This is true for some associative algebras and for all two-sided unital algebras (see  \cref{thm:bijection_left_unital}\ref{item:HAstr_B4}).

		\begin{theorem}\label{thm:ACL_split}
			Let $(A,1_l)$ be a left-unital algebra. It holds that
			\begin{equation}\label{eqn:split-of-ACL}
				AC_l(A)=AC_l(A)\cdot 1_l\oplus Ann_A^l(A).
			\end{equation}
			Moreover, $AC_l(A)\cdot 1_l=AC_l(A)$ if and only if $Ann_A^l(A)=\{0\}$.
		\end{theorem}
		\begin{proof}
			First, both $AC_l(A)\cdot 1_l$ and $Ann_A^l(A)$ are subspaces of $AC_l(A)$, since any $a\in Ann_A^l(A)$ trivially makes $(A,L_a,1_l)$ hom-associative.
			Next, observe that the property $L_a=L_{a1_l}$ reads, for all $a\in AC_l(A)$, $L_{a1_l-a}=0$, or equi\-valently, $a\cdot 1_l-a\in Ann_A^l(A)$. By linearity, $a\cdot 1_l\in Ann_A^l(A)$ if and only if $a\in Ann_A^l(A)$. But then $a\cdot 1_l=0$. If $a\in AC_l(A)\cdot 1_l$, then it commutes with $1_l$, which implies $a=1_l\cdot a=a\cdot 1_l=0$. That is, $AC_l(A)\cdot 1_l$ and $Ann_A^l(A)$ are disjoint subspaces. Finally, we can write every $a\in AC_l(A)$ uniquely as
			\begin{equation*}
				a=a\cdot 1_l+(a-a\cdot 1_l)\in AC_l(A)\cdot 1_l\oplus Ann_A^l(A).
			\end{equation*}
			The sum is direct. This completes the proof.
		\end{proof}
		\begin{example}
			The unity $1_{lr}$ of any two-sided unital algebra is a two-sided regular element. We then have $Ann_A^l(A)=\{0\}$, and thus $AC_l(A)\cdot 1_{lr}=AC_l(A)$.
		\end{example}
		
		The two-sided unital $AC(A)$ is a subalgebra of $A$. It is then highly desirable to find a subspace of $AC_l(A)$ that is also a subalgebra.
		\begin{remark}
			In \cref{thm:ACL1_subalgebra}, we show that $AC_l(A)\cdot 1_l$ is indeed a subalgebra of $A$. This is key to many of the results in this section.
		\end{remark}
		
		Left-unital associative algebras can be seen as hom-unital hom-associative algebras with hom-unity $1_l$. That is, $1_l\in AC_l(A)$ if and only if $(A,1_l)$ is associative.

		\begin{theorem}\label{thm:bijection_left_unital}
			Let $(A,1_l)$ be a left-unital algebra. The following statements are true:
			\begin{enumerate}[label=\textnormal{(B\arabic*):},ref=\textnormal{(B\arabic*)}]
				\item\label{item:HAstr_B1} The subalgebra $Twist(A)$ of all HA-compatible maps on $(A,1_l)$ is in bijection with the commutative associative subalgebra
				\begin{equation*}
					AC_l(A)\cdot 1_l=\{a\in AC_l(A)\mid a\cdot 1_l=a\}.
                \end{equation*}
				The bijection is given by $\phi$ and $\psi\vert_{AC_l(A)\cdot 1_l}$, where:
				\begin{align*}
					\begin{array}{ccc}
						\phi:Twist(A)& \to &   A    \\
						\alpha & \mapsto & \alpha(1_l)
					\end{array}\quad  ,
					&&
					\begin{array}{ccc}
						\psi:AC_l(A)& \to &  \mathcal{L}(A)   \\
						a & \mapsto & L_a\colon x\mapsto a\cdot x.
					\end{array}
				\end{align*}
				\item\label{item:HAstr_B2} For every $a\in AC_l(A)$,  $(A,L_a,1_l)$ is hom-associative.
				\item\label{item:HAstr_B3} Multiplicative HA-compatible maps on $(A,1_l)$ are in bijection with the subset $E(A)\cap AC_l(A)\cdot 1_l$ of idempotents in $AC_l(A)\cdot 1_l$.
				\item\label{item:HAstr_B4} If $(A,1_l)$ is moreover associative, then HA-compatible maps of $(A,1_l)$ are in bijection with the commutative associative subalgebra
				\begin{equation}
					AC_l(A)=\{a\in A \mid a\cdot x\cdot z=x\cdot a\cdot z\ \forall x,z\in A \}.
				\end{equation}
				The bijection is given by $a\leftrightarrow L_a$. Moreover, $AC_l(A)$ satisfies the following properties.
				\begin{enumerate}[label=\textnormal{\arabic*)}]
					\item\label{item:HAstr_B41} $Z(A)\subseteq AC_l(A)$.
					\item\label{item:HAstr_B42} If $(A,1_l)$ is two-sided unital, a domain or commutative, then it holds that $AC_l(A)=Z(A)=AC_l(A)\cdot 1_l=HU_n^l(A)$.
				\end{enumerate}
			\end{enumerate}
		\end{theorem}

		\begin{proof}
			To prove \ref{item:HAstr_B1} and \ref{item:HAstr_B2}, that is, the bijection between $\phi$ and a restriction of $\psi$, we show that they are bijective and mutually inverse, that is:
			\begin{enumerate}
				\item $(\psi\circ\phi)(\alpha)=\alpha$ for all $\alpha\in Twist(A)$.
				\item $(\phi\circ\psi)(a)=a$ for all $a\in AC_l(A)\cdot 1_l$.
			\end{enumerate}
			Both of these fail if the corresponding map fails to be injective.
			
			Applying $\psi$ to $a\in AC_l(A)$, it follows that:
			\begin{enumerate}
				\item $(A,L_a,1_l)$ is hom-associative. That is, $\psi(a)=L_a\in Twist(A)$.
				\item $(\phi\circ\psi)(a)=\phi(L_a)=L_a(1_l)=a\cdot 1_l\in AC_l(A)\cdot 1_l$.
				\item The zero map is always HA-compatible, and thus
				\begin{equation*}
					\ker\psi=\{a\in AC_l(A)\mid L_a=0\}=Ann_ A ^l(A)\cap AC_l(A)=Ann_ A ^l(A).
				\end{equation*}
				
			\end{enumerate}
			Using the last property in \cref{tab:product_relations_one}, for $a\in AC_l(A)$ we have $L_{a1_l}=L_a$, that is, $\psi(a1_l)=\psi(a)$. By \cref{thm:ACL_split}, if we restrict $\psi$ to $AC_l(A)\cdot 1_l$ we obtain
			\begin{equation*}
				\ker\psi\vert_{AC_l(A)1_l}=\{a\in AC_l(A)\cdot 1_l\mid L_a{=}0\}=AC_l(A)\cdot 1_l\cap Ann_ A ^l(A)=\{0\}.
			\end{equation*}
			That is, $\psi\vert_{AC_l(A)\cdot 1_l}$ is injective. Moreover, for any $a\in AC_l(A)\cdot 1_l$ it holds that $(\phi\circ\psi)(a)=\phi(L_a)=a\cdot 1_l=a$.

			Let $\alpha\in\ker\phi$. By the hom-associativity of $(A,\alpha,1_l)$ we have $\alpha=L_{\alpha(1_l)}=L_0=0$. It follows that $\phi$ is injective.

			Now, let $a\in\phi(Twist(A))$. Then, $a=\alpha(1_l)$ for some $\alpha$ such that $(A,\alpha,1_l)$ is hom-associative. We have the following:
			\begin{enumerate}
				\item $L_{xy}\circ L_a=L_{ax}\circ L_y$ \textbf{(hom-associativity).}
				\item $x\cdot (a\cdot y)=x\cdot (\alpha(1_l)\cdot y)=x\cdot \alpha(y)=\alpha(x\cdot y)=\alpha(1_l)\cdot (x\cdot y){=}a\cdot (x\cdot y)$ that is, $L_x\circ L_a=L_a\circ L_x$ \textbf{(Property \ref{item:ACL1}).}
				\item Putting the previous two together, $L_a\circ L_{xy}=L_{ax}\circ L_y$ \textbf{(Property \ref{item:ACL2})}.
				\item By left unitality (see \cref{tab:product_relations_one}), we have $[a,1_l]=0$, that is, $a\cdot 1_l=a$.
			\end{enumerate}
			From these properties, it follows that $a\in AC_l(A)\cdot 1_l$. Now, by left unitality, it holds that $(\psi\circ\phi)(\alpha)=\psi(\alpha(1_l))=L_{\alpha(1_l)}=L_a$. We have that $a\leftrightarrow L_a$ is indeed an isomorphism between $Twist(A)$ and $AC_l(A)\cdot 1_l$.
			
			Finally, that $Twist(A)$ is a subalgebra of $\mathcal{L}(A)$ follows from the multiplicativity of $\phi$ and $\psi$ and $AC_l(A)\cdot 1_l$ being itself a subalgebra of $A$. Indeed:
			\begin{enumerate}[label=,leftmargin=*]
				\item \textbf{Multiplicativity of $\phi\colon$}
				\begin{equation*}
					\phi(\alpha\circ\alpha')=(\alpha\circ\alpha')(1_l)=\alpha(\alpha'(1_l))=\alpha(1_l)\cdot\alpha'(1_l)=\phi(\alpha)\cdot \phi(\alpha').
				\end{equation*}
				
				\item\textbf{Multiplicativity of $\psi\colon$} For all $a,a' \in AC_l(A)$ we have:
				\vspace{-0.25cm}
				\begin{multline*}
					\psi(a\cdot a')(1_l\cdot x)=L_{aa'}(1_l\cdot x)=(a\cdot a')\cdot (1_l\cdot x)=(a'\cdot 1_l)\cdot (a\cdot x)\\
					=(1_l\cdot a)\cdot (a'\cdot x)
					=a\cdot (a'\cdot x)=(L_{a}\circ L_{a'})(x)=\left(\psi(a)\circ\psi(a')\right)(x).
				\end{multline*}
			\end{enumerate}

			To prove \ref{item:HAstr_B3}, observe that, by \ref{item:HAstr_B1}, any HA-compatible map is multiplication by an element $a\in AC_l(A)\cdot 1_l$. By \cref{thm:unital_multiplicativity}, such a map $L_a$ is multiplicative if and only if $a$ is idempotent.

			Lastly, in order to prove \ref{item:HAstr_B4}, observe firstly that Property \ref{item:ACL2} is a consequence of Property \ref{item:ACL1}. Indeed, $(a\cdot x)\cdot (y\cdot z)=a\cdot x\cdot y\cdot z=a\cdot ((x\cdot y)\cdot z)$. It follows that
			\begin{equation*}
				AC_l(A)=\{a\in A\mid a\cdot x\cdot z=x\cdot a\cdot z,\ \forall x,z\in A\}.
			\end{equation*}
			Let $a\in AC_l(A)$. Then, for all $x,y,z\in A$, we have
			\begin{equation*}
				L_a(x)\cdot (y\cdot z)=(a\cdot x)\cdot (y\cdot z)=a\cdot (x\cdot y)\cdot z=(x\cdot y)\cdot a\cdot z=(x\cdot y)\cdot L_a(z).
			\end{equation*}
			Thus, $(A,L_a,1_l)$ is hom-associative. Conversely, let $\alpha$ be HA-compatible, that is, let $(A,\alpha,1_l)$ be hom-associative. Then,
			\begin{enumerate}
				\item By left unitality (see \cref{tab:product_relations_one}), $\alpha=L_{a}$, where $a=\alpha(1_l)$ .
				\item By hom-associativity, $(a\cdot x)\cdot y\cdot z=x\cdot y\cdot (a\cdot z)$. By contracting $x\mapsto 1_l$, we obtain $a\cdot y\cdot z=y\cdot a\cdot z$, that is, $a\in AC_l(A)$.
			\end{enumerate}
			
            Lastly, we prove \ref{item:HAstr_B41} and \ref{item:HAstr_B42}. Let $(A,1_l)$ be associative and $a\in Z(A)$. For all $x,z\in A$, element $a$ satisfies
			\begin{equation}\label{eqn:ZACl_associative}
				a\cdot x=x\cdot a\Rightarrow (a\cdot x-x\cdot a)\cdot z=0\Rightarrow a\cdot x\cdot z=x\cdot a\cdot z\Rightarrow a\in AC_l(A).
			\end{equation}
			Thus, \ref{item:HAstr_B41} is proved. If $(A,1_l)$ is an associative domain, then the reciprocal is also true by the cancellation law. 	
			If $(A,1_l)$ is two-sided unital, then by the construction of $AC(A)$ (see \cite{FreGohr:UnitCond}), we have $HU_n^l(A)=AC(A)\subseteq Z(A)$. If $(A,1_l)$ is also associative, then applying \eqref{eqn:ZACl_associative} yields $HU_n^l(A)=AC(A)=Z(A)$. If $(A,1_l)$ is commutative, then it is two-sided unital. This proves \ref{item:HAstr_B42} and completes the proof.

		\end{proof}
		
		\begin{remark}
			Observe that $E(A)$ is not a linear subspace: for $e_1,\dots,e_n\in E(A)$, we have
			\begin{equation*}
				(e_1+\dots+e_n)^2=\sum\limits_{i=1}^n e_i + \sum\limits_{i\neq j}e_ie_j=\sum\limits_{i=1}^n e_i + \sum\limits_{i< j}\{e_i,e_j\},
			\end{equation*}
			where $\{e_i,e_j\}=e_ie_j+e_je_i$ is an anti-commutator. Then $(e_1+\dots+e_n)^2=e_1+\dots+e_n$ if and only if $\sum\limits_{i< j}\{e_i,e_j\}=0$. This can happen if the $e_i$ anti-commute, but in the general case, $span(E(A))$ is not entirely made of idempotents. This supports the intuitive idea that a linear combination of multiplicative linear maps, given by the product of their coordinate matrices, need not be multiplicative.
		\end{remark}

		\subsection{Element-wise relations in \texorpdfstring{$AC_l(A)$}{AC\_l(A)}}\label{ssec:relations-NA-ACL}
		
		The subspace $AC_l(A)$ satisfies several commutation and association relations as a consequence of the hom-associativity axiom.

		\begin{lemma}\label{lem:ACL1_properties_A}
			Every $a\in AC_l(A)$ satisfies, for all $x,y,z\in A$, the properties listed in \cref{tab:ACL}.
			\begin{table}[H]
				\centering
				\caption{Elementary properties of elements in $AC_l(A)$.}
				\label{tab:ACL}
				\begin{NiceTabular}{|>{\ifnum\value{iRow}=0\else \refstepcounter{propertycounter}L\arabic{propertycounter}\fi}c|cc|}[first-row]
					& Element form & Operator form \\
					\hline
					\label{item:ACL1} & $a\cdot (x\cdot y)=x\cdot (a\cdot y)$ & $L_a\circ L_x=L_x\circ L_a$ \\ \hline
					\label{item:ACL2} & $(a\cdot x)\cdot (y\cdot z)=a\cdot ((x\cdot y)\cdot z)$ & $L_{ax}\circ L_y=L_a\circ L_{xy}$ \\ \hline
					\label{item:ACL3} & $(a\cdot 1_l)\cdot x=a\cdot x$ & $L_{a1_l-a}=0$ \\ \hline
					\label{item:ACL4} & $(a\cdot 1_l)\cdot 1_l=a\cdot 1_l$ & $[a1_l,1_l]=0$ \\ \hline
					\label{item:ACL1R} & $a\cdot (x\cdot 1_l)=x\cdot (a\cdot 1_l)$ & $L_a\circ R_{1_l}=R_{a1_l}$ \\ \hline
					\label{item:ACL2RR} & $(a\cdot x)\cdot (y\cdot z)=a\cdot ((x\cdot y)\cdot z)$ & $L_{ax}\circ R_z=L_a\circ R_z\circ L_x$ \\ \hline
					\label{item:ACL3RR} & $(a\cdot x)\cdot 1_l=a\cdot (x\cdot 1_l)$ &  $R_{1_l}\circ L_a=L_a\circ R_{1_l}$ \\ \hline
					\label{item:ACL4RR} & $x\cdot (a\cdot 1_l)=(x\cdot 1_l)\cdot (a\cdot 1_l)=a\cdot (x\cdot 1_l)$ & $R_{a1_l}=R_{a1_l}\circ R_{1_l}=L_a\circ R_{1_l}$ \\ \hline
				\end{NiceTabular}
			\end{table}
		\end{lemma}
		
		\begin{proof}
			Properties \ref{item:ACL1} and \ref{item:ACL2} are the definition of $AC_l(A)$.
			
			Property \ref{item:ACL3} is an immediate consequence of \ref{item:ACL2}: contracting $x\mapsto 1_l$ yields
			\begin{equation*}
				(a\cdot 1_l)\cdot(y\cdot z)=a\cdot ((1_l\cdot y)\cdot z)=a\cdot (y\cdot z).
			\end{equation*}
			
			Property \ref{item:ACL4} follows from contracting $y,z\mapsto 1_l$ in \ref{item:ACL3}.
			
			Contracting $y\mapsto 1_l$ in \ref{item:ACL1} yields \ref{item:ACL1R}: $L_a\circ R_y=R_{ay}\overset{y\mapsto 1_l}{\longrightarrow}L_a\circ R_{1_l}=R_{a1_l}$.
			
			Rewriting  \ref{item:ACL2} in terms of
			left and right multiplication operators yields \ref{item:ACL2RR}:
			\begin{align*}
				(L_{ax}\circ L_y)(z)&=(a\cdot x)\cdot (y\cdot z)=(L_{ax}\circ R_z)(y)\\
				(L_a\circ L_{xy})(z)&=a\cdot ((x\cdot y)\cdot z)=(L_a\circ R_z\circ L_x)(y)
			\end{align*}
			
			Properties \ref{item:ACL3RR} and \ref{item:ACL4RR} are obtained by rewriting \ref{item:ACL2RR} using left and right multiplication operators:
			\begin{multline*}
				(a\cdot x)\cdot (y\cdot z)=a\cdot ((x\cdot y)\cdot z)=(x\cdot y)\cdot (a\cdot z)\\
				\overset{z\mapsto 1_l}{\longrightarrow} (a\cdot x)\cdot (y\cdot 1_l)=a\cdot ((x\cdot y)\cdot 1_l)=(x\cdot y)\cdot (a\cdot 1_l)\\
				\longrightarrow R_{y1_l}\circ L_a=L_a\circ R_{1_l}\circ R_y=R_{a1_l}\circ R_y\\
				\overset{y\mapsto 1_l}{\longrightarrow} R_{1_l}\circ L_a=L_a\circ R_{1_l}\circ R_{1_l}=R_{a1_l}\circ R_{1_l}\overset{\ref{item:ACL4}}{=}R_{a1_l}\overset{\ref{item:ACL1R}}{=}L_a\circ R_{1_l}.
			\end{multline*}
			By grouping terms, we can read this equality in different ways, namely  $R_{1_l}\circ L_a=L_a\circ R_{1_l}$ (Property \ref{item:ACL3RR}) and $R_{a1_l}=R_{a1_l}\circ R_{1_l}=L_a\circ R_{1_l}$ (Property \ref{item:ACL4RR}).
		\end{proof}
		\begin{remark}\label{rem:ACL1_assoc}
			All these properties can be stated in terms of elements instead of multiplication operators, yielding an array of product relations. We are particularly interested in \textnormal{\ref{item:ACL3RR}}, which reads $[a,x,1_l]_{\mathrm{as}}=0$ for all $x\in A$.
		\end{remark}

		Observe that properties \ref{item:ACL3}, \ref{item:ACL4}, \ref{item:ACL1R} and \ref{item:ACL4RR} are written in terms of the element $a\cdot 1_l$. This is quite natural, as shown in \cref{thm:ACL_split}. Property \ref{item:ACL2RR} can also be contracted into one of these.

		\begin{proposition}\label{prop:ACL1_properties_b}
			Every element $b\in AC_l(A)\cdot 1_l$ satisfies, for all $b'\in AC_l(A)\cdot 1_l$ and $x\in A$, the properties in \cref{tab:ACL_1}.
			\begin{table}[ht!]
				\centering
				\caption{Properties of elements in $AC_l(A)\cdot 1_l$.}
				\label{tab:ACL_1}
				\begin{NiceTabular}{|>{\ifnum\value{iRow}=0\else \refstepcounter{propertycounter}L\arabic{propertycounter}\fi}c|cc|}[first-row]
					& Element form & Operator form \\
					\hline\label{item:ACL1com1x} & $b\cdot (x\cdot 1_l)=x\cdot b$ & $L_b\circ R_{1_l}=R_b$ \\ \hline
					\label{item:ACL1comx1} & $b\cdot (x\cdot 1_l)=(x\cdot 1_l)\cdot b$ & $L_b\circ R_{1_l}=R_b\circ R_{1_l}$ \\ \hline
					\label{item:ACL1pseudoCom} & $(b\cdot x)\cdot 1_l=(x\cdot 1_l)\cdot b$ & $R_{1_l}\circ L_b=R_b\circ R_{1_l}$ \\ \hline
					\label{item:ACL1x1inv} & $x\cdot b=(x\cdot 1_l)\cdot b$ & $R_b\circ R_{1_l}=R_b$ \\ \hline
					\label{item:ACL1_as_R} & $[b,b',x]_{\mathrm{as}}=0$ & $L_{bb'}=L_b\circ L_{b'}$ \\ \hline
					\label{item:ACL1_as_M} & $[b,x,b']_{\mathrm{as}}=0$ & $L_b\circ R_{b'}=R_{b'} \circ L_b$ \\ \hline
					\label{item:ACL1_as_L} & $[x,b,b']_{\mathrm{as}}=0$ & $R_{bb'}=R_{b'}\circ R_b$ \\ \hline
				\end{NiceTabular}
			\end{table}
			
		\end{proposition}
		\begin{proof}
			Properties \ref{item:ACL1com1x}--\ref{item:ACL1x1inv} are a rewriting of \ref{item:ACL1}--\ref{item:ACL4}. Properties \ref{item:ACL1_as_R}--\ref{item:ACL1_as_L} follow from the interactions of properties in Lemma \ref{prop:ACL1_properties_b} with products of two elements $b,b'\in AC_l(A)\cdot 1_l$. Indeed,
			\begin{align*}
				\ref{item:ACL1comx1}&\colon L_{bb'}\circ R_{1_l}\overset{\ref{item:ACL2RR}}{=} L_b\circ R_{1_l}\circ R_{b'}\overset{\ref{item:ACL1R}}{=} R_b\circ R_b'\Rightarrow (b\cdot b')\cdot (x\cdot 1_l)=(x\cdot b')\cdot b\\
				&\hspace{1cm}=(x\cdot b')\cdot (b\cdot 1_l)\overset{\ref{item:ACL2}}{=}(b\cdot x)\cdot (b'\cdot 1_l)\overset{\ref{item:ACL1}}{=}b'\cdot ((b\cdot x)\cdot 1_l).\textnormal{ Use Remark \ref{rem:ACL1_assoc}: }\\
				&\hspace{1cm}=b'\cdot ((b\cdot x)\cdot 1_l)=b'\cdot (b\cdot (x\cdot 1_l))\overset{\ref{item:ACL1}}{=}b\cdot (b'\cdot (x\cdot 1_l))\overset{\ref{item:ACL1comx1}}{=}b\cdot ((x\cdot 1_l)\cdot b')\\
				&\hspace{1cm}\overset{\ref{item:ACL1}}{=}(x\cdot 1_l)\cdot (b\cdot b')\Rightarrow (b\cdot b')\cdot (x\cdot 1_l)=(x\cdot 1_l)(b\cdot b')\\
				\ref{item:ACL1pseudoCom}&\colon \textnormal{ follows from applying }\ref{item:ACL3RR}\textnormal{ to }\ref{item:ACL1comx1}.\\
				\ref{item:ACL1x1inv}&\colon x(bb')\overset{\ref{item:ACL1}}{=}b\cdot (x\cdot b')\overset{\ref{item:ACL1pseudoCom}}{=}b\cdot ((x\cdot 1_l)\cdot b')\overset{\ref{item:ACL1}}{=}(x\cdot 1_l)\cdot (b\cdot b').\\
				\ref{item:ACL1com1x}&\colon \textnormal{ follows from }\ref{item:ACL1comx1} \textnormal{ and }\ref{item:ACL1x1inv}.
			\end{align*}
			This completes the proof.
		\end{proof}
		
		\begin{theorem}\label{thm:ACL1_subalgebra}
			Let $(A,1_l)$ be a left-unital algebra. $AC_l(A)\cdot 1_l$ is a commutative associative subalgebra of $(A,1_l)$.
			Moreover, $\left(AC_l(A)\cdot 1_l\right)^2\subseteq Ann_A^l([A,A,A]_{\mathrm{as}})$.
		\end{theorem}
		\begin{proof}
			Let $b,b'\in AC_l(A)\cdot 1_l$. It has already been proven that $b\cdot b'$ commutes with $1_l$, so it suffices to examine whether $b\cdot b'\in AC_l(A)$, that is, whether $(A,L_{bb'},1_l)$ is hom-associative. First, observe that elements of the form $b\cdot b'$ still satisfy Property \ref{item:ACL1}, as a consequence of $b$ and $b'$ satisfying it:
			\begin{equation*}
				(b\cdot b')\cdot (x\cdot y)\overset{\ref{item:ACL1_as_R}}{=}b\cdot (b'\cdot (x\cdot y))\overset{\ref{item:ACL1}}{=}b\cdot (x\cdot (b'\cdot y))\overset{\ref{item:ACL1}}{=}x\cdot (b\cdot (b'\cdot y))=x\cdot ((b\cdot b')\cdot y).
			\end{equation*}
			The hom-associator $[x,y,z]_{\mathrm{as}}^{L_{bb'}}$ of any three elements of $A$ is as follows:
			\begin{align*}
				& ((b\cdot b')\cdot x)\cdot (y\cdot z)\overset{\ref{item:ACL1}}{=}x\cdot ((b\cdot b')\cdot (y\cdot z))\overset{\ref{item:ACL1_as_R}}{=}x\cdot (b\cdot (b'\cdot (y\cdot z)))\\
				& \hspace{2.1cm} \overset{\ref{item:ACL1}}{{=}}b\cdot (x\cdot (b'\cdot (y\cdot z)))\overset{\ref{item:ACL1}}{=}b\cdot (b'\cdot (x\cdot (y\cdot z)))\overset{\ref{item:ACL1_as_L}}{=}(b\cdot b')\cdot (x\cdot (y\cdot z)),\\
				& (x\cdot y)\cdot ((b\cdot b')\cdot z)\overset{\ref{item:ACL1_as_R}}{=}(x\cdot y)\cdot (b\cdot (b'\cdot z))\overset{\ref{item:ACL1}}{=}b\cdot ((x\cdot y)\cdot (b'\cdot z))\\
				& \hspace{2.8cm} \overset{HA}{=} b\cdot ((b'\cdot x)\cdot (y\cdot z))\overset{\ref{item:ACL2}}{=}(b\cdot b')\cdot (x\cdot (y\cdot z)).
			\end{align*}
			By taking the difference, it follows that $[x,y,z]_{\mathrm{as}}^{L_{bb'}}=0$. Observe that we can use Property \ref{item:ACL1} on $(x\cdot y)\cdot ((b\cdot b')\cdot z)$ instead, obtaining
			$(b\cdot b')\cdot ((x\cdot y)\cdot z)$. It follows that $(b\cdot b')\cdot ([x,y,z]_{\mathrm{as}})=[x,y,z]_{\mathrm{as}}^{L_{bb'}}=0$. This shows that $AC_l(A)\cdot 1_l$ is indeed a subalgebra and $b\cdot b'\in Ann_A^l([A,A,A]_{as})$.
			
			By Property \ref{item:ACL1}, we have
			\begin{equation*}
				b\cdot b'=(b\cdot b')\cdot 1_l=b\cdot (b'\cdot 1_l)\overset{\ref{item:ACL1}}{=}b'\cdot (b\cdot 1_l)=(b'\cdot b)\cdot 1_l=b'\cdot b.
			\end{equation*}
			Thus, $AC_l(A)\cdot 1_l$ is commutative. Finally, properties \ref{item:ACL1_as_R}--\ref{item:ACL1_as_L} indicate that $AC_l(A)\cdot 1_l$ is associative. This completes the proof.
		\end{proof}
		\begin{example}
			A left-unital algebra $(A,1_l)$ without proper subalgebras satisfies one of the following:
			\begin{enumerate}[label=\textnormal{\arabic*)}]
				\item $AC_l(A)\cdot 1_l=\{0\}$, and it has no non-trivial hom-associative structures.
				\item $AC_l(A)\cdot 1_l=A$, and $A$ is two-sided unital associative. Indeed, it follows from $1_l\in AC_l(A)\cdot 1_l$ that $A$ is associative. Moreover, all elements of $A$ are in $AC_l(A)\cdot 1_l$ and thus commute with $1_l$, which shows that $1_l$ is two-sided.
			\end{enumerate}
		\end{example}
		
		Association relations \ref{item:ACL1_as_R}, \ref{item:ACL1_as_M} and \ref{item:ACL1_as_L} involve two elements in $AC_l(A)\cdot 1_l$, and thus we do not have $AC_l(A)\subseteq N(A)$. This marks a clear difference with the two-sided unital case: $AC(A)$ is, by definition, a subspace of $N(A)$, as shown in \cref{thm:AC_two_sided_unital}.
		Using partial nucleus notation, we can write these relations as
		\begin{equation*}
			N_ A
			^l(AC_l(A)\cdot 1_l)=N_ A ^m(AC_l(A)\cdot 1_l)=N_ A ^r(AC_l(A)\cdot 1_l)= A .
		\end{equation*}
		
		\begin{proposition}\label{prop:ACl_2_nilpotent}
			Let $(A,1_l)$ be a left-unital algebra, $x,y,z\in A$ such that $[x,y,z]_{\mathrm{as}}$ is right-regular. The following statements hold:
			\begin{enumerate}[label=\textnormal{\arabic*)}]
				\item For all $b,b'\in AC_l(A)\cdot 1_l,\ b\cdot b'=0$.
				\item All elements in $AC_l(A)\cdot 1_l$ are 2-nilpotent.
				\item For all $b\in AC_l(A)\cdot 1_l$ and $x,y,z\in A$, $b\cdot [x,y,z]_{\mathrm{as}}$ is a right zero divisor.
			\end{enumerate}
		\end{proposition}
		\begin{proof}
			\cref{thm:ACL1_subalgebra} indicates that $b\cdot b'$ is a left zero divisor of all associators of $A$ for any $b,b'\in AC_l(A)\cdot 1_l$. If there exist $x,y,z\in A$ such that $[x,y,z]_{\mathrm{as}}$ is right-regular, then $b\cdot b'=0$. In particular, every element $b\in AC_l(A)\cdot 1_l$ is 2-nilpotent. Applying relation \ref{item:ACL1_as_R} to \cref{thm:ACL1_subalgebra} yields $b\cdot (b'\cdot [x,y,z]_{\mathrm{as}})=0$, and in particular $b\cdot (b\cdot [x,y,z]_{\mathrm{as}})=0$ for all $b\in AC_l(A)\cdot 1_l$.
			\begin{enumerate}[label=\textnormal{\arabic*)}]
				\item If $AC_l(A)\cdot 1_l\neq\{0\}$, then $b\cdot [x,y,z]_{\mathrm{as}}$ is a right zero divisor of $b$.
				\item If $AC_l(A)\cdot 1_l=\{0\}$, then $b\cdot [x,y,z]_{\mathrm{as}}=0$, which is trivially a global zero divisor.
			\end{enumerate}
			This completes the proof.
		\end{proof}
		
		\begin{example}[Division algebras]
			We can consider one-sided unital division algebras as domains where all elements are invertible (if associative) or all multiplication operators are bijective (if non-associative). In the non-associative case, Proposition \ref{prop:ACl_2_nilpotent} implies that $b\cdot (b\cdot [x,y,z]_{\mathrm{as}})=0$ for all $b\in AC_l(A)\cdot 1_l$ and $x,y,z\in A$. This relation indicates that $L_b$ can never be bijective. Thus,  $(A,L_b,1_l)$ cannot be a division algebra. There are no nontrivial hom-associative one-sided unital division algebras. Note that this includes two-sided unital division algebras.
		\end{example}

		\begin{theorem}\label{thm:no_domains_L}
			Let $(A,1_l)$ be a domain. Then, $AC_l(A)= Z(A)$ if $(A,1_l)$ is associative, and $AC_l(A)\cdot 1_l=\{0\}$ otherwise. Moreover, the following statements are true:
			\begin{enumerate}[label=\textnormal{\arabic*)},ref=\textnormal{\arabic*)}]
				\item Any left-unital hom-associative domain is either trivial or associative.
				\item Any two-sided unital hom-associative domain is either trivial or associative.
				\item Any two-sided unital hom-associative division algebra is either trivial or associative.
			\end{enumerate}
		\end{theorem}
		\begin{proof}
			Let $(A,1_l)$ be an associative domain.  That $AC_l(A)= Z(A)$ is proven in \cref{thm:bijection_left_unital}\ref{item:HAstr_B4}. If $(A,1_l)$ is a non-associative domain, then all multiplication operators are injective. Thus, for any $b\in AC_l(A)\cdot 1_l$ it holds that $b\cdot b=L_b(b)=0\Rightarrow b=0$. Consequently, a left-unital hom-associative domain with $AC_l(A)\cdot 1_l\neq\{0\}$ can only be associative or trivial, in which case $AC_l(A)=A$. Finally, any two-sided unital algebra is left-unital, and any division algebra is a domain. This completes the proof.
		\end{proof}

		\subsection{The right-sided counterpart}\label{ssec:right-sided-unital}
		
		In this chapter, we have considered left-unital algebras. This approach is motivated by the fact that the opposite algebra of a given right-unital (left-unital) algebra  is left-unital (right-unital).
		
		On a given right-unital algebra, the subspace of hom-unities is given by
		\begin{equation*}
			AC_r(A)=\{a\in A \mid R_a\circ R_y=R_y\circ R_a,\ R_{za}\circ R_{y}=R_{a}\circ R_{yz},\ \forall y,z\in A  \}.
		\end{equation*}
		\begin{proposition}
			For any right-unital algebra $(A,1_r)$ it holds that
			\begin{equation*}
				\{a\in AC_r(A)\mid a=1_r\cdot a\}=1_r\cdot AC_r(A)\subseteq AC_r(A).
			\end{equation*}
			Equality is reached when $Ann_ A ^r(A)=\{0\}$. Moreover,
			\begin{equation*}
				AC_r(A)=1_r\cdot AC_r(A)\oplus Ann_A^r(A).
			\end{equation*}
		\end{proposition}
		
		\begin{theorem}
			$1_r\cdot AC_r(A)$ is a commutative associative subalgebra of $ A$. Moreover, $\left(1_r\cdot AC_r(A)\right)^2\subseteq Ann_ A ^r([ A, A, A ]_{\mathrm{as}})$.
		\end{theorem}
		
		\begin{theorem}
			Let $(A,1_r)$ be a right-unital algebra and $E(A)$ the set of all idempotents of $ A$. The following statements hold:
			\begin{enumerate}[label=\textnormal{\arabic*)}]
				\item Twisting maps $\beta$ such that $(A,\beta,1_r)$ is hom-associative are in bijection with the commutative associative subalgebra
				\begin{equation}
					1_r\cdot AC_r(A)=\{a\in AC_r(A)\mid  a=1_ra\}.
				\end{equation}
				The bijection is given by $\phi$ and $\psi\vert_{1_r\cdot AC_r(A)}$.
				\item For every $a\in AC_r(A)$, the product in $A$ is hom-associative with twisting map $R_a=R_{1_ra}$.
				\item If $Z_ A (1_r)=\{0\}$, then $(A,1_r)$ has no non-trivial hom-associative structures.
				\item All multiplicative HA-compatible maps of $(A,1_r)$ have the form
				\begin{equation*}
					\{R_e\mid e\in E(A)\cap 1_r\cdot AC_r(A)\}    .
				\end{equation*}
			\end{enumerate}
		\end{theorem}
		The elements of $1_r\cdot AC_r(A)$ behave as follows:
		\begin{lemma}
			Every element $a\in AC_r(A)$ satisfies, for all $x,y,z\in A$, the properties listed in \cref{tab:ACR}.
			\setcounter{propertycounterright}{0}
			\begin{table}[ht!]
				\centering
				\caption{Elementary properties of elements in $AC_r(A)$.}
				\label{tab:ACR}
				\begin{NiceTabular}{|>{\ifnum\value{iRow}=0\else \refstepcounter{propertycounterright}R\arabic{propertycounterright}\fi}c|cc|}[first-row]
					& Element form & Operator form \\
					\hline
					& $(x\cdot y)\cdot a=(x\cdot a)\cdot y$ & $R_a\circ R_y=R_y\circ R_a$ \\ \hline
					& $ (x\cdot y)\cdot (z\cdot a)=(x\cdot (y\cdot z))\cdot a $ & $R_{za}\circ R_y=R_a\circ R_{yz}$ \\ \hline
					& $x\cdot (1_r\cdot a)=x\cdot a$ & $R_{(1_ra-a)}=0$ \\ \hline
					& $1_r\cdot (1_r\cdot a)=1_r\cdot a$ & $[1_ra,1_r]=0$ \\ \hline
					& $(1_r\cdot x)\cdot a=(1_r\cdot a)\cdot x$ & $R_a\circ L_{1_r}=L_{1_ra}$ \\ \hline
					& $(x\cdot a)\cdot (y\cdot z)=(x\cdot (y\cdot z))\cdot a$ & $L_{xa}\circ R_z=R_a\circ L_x\circ R_z$ \\ \hline
					& $1_r\cdot (x\cdot a)=(1_r\cdot x)\cdot a$ &  $L_{1_r}\circ R_a=R_a\circ L_{1_r}$ \\ \hline
					& $(1_r\cdot a)\cdot x=(1_r\cdot a)\cdot (1_r\cdot x)=(1_r\cdot x)\cdot a$ & $L_{1_ra}=L_{1_ra}\circ L_{1_r}=R_a\circ L_{1_r}$ \\ \hline
				\end{NiceTabular}
				
			\end{table}
		\end{lemma}
		
		\begin{lemma}
			Every element $b\in 1_r\cdot AC_r(A)$ satisfies, for all $b'\in 1_r\cdot AC_r(A)$ and $x,y,z\in A$, the properties listed in \cref{tab:ACR_1}.
			\begin{table}[H]
				\centering
				\caption{Properties of elements in $1_r\cdot AC_r(A)$.}
				\label{tab:ACR_1}
				\begin{NiceTabular}{|>{\ifnum\value{iRow}=0\else \refstepcounter{propertycounterright}R\arabic{propertycounterright}\fi}c|cc|}[first-row]
					& Element form & Operator form \\
					\hline & $(1_r\cdot z)\cdot b=b\cdot z$ & $R_b\circ L_{1_r}=L_b$ \\ \hline
					& $(1_r\cdot z)\cdot b=b\cdot (1_r\cdot z)$ & $R_b\circ L_{1_r}=L_b\circ L_{1_r}$ \\ \hline
					& $1_r\cdot (z\cdot b)=b\cdot (1_r\cdot z)$ & $L_{1_r}\circ R_b=L_b\circ L_{1_r}$ \\ \hline
					& $b\cdot z=b\cdot (1_r\cdot z)$ & $L_b\circ L_{1_r}=L_b$ \\ \hline
					& $[b,b',z]_{\mathrm{as}}=0$ & $L_{bb'}=L_b\circ L_{b'}$ \\ \hline
					& $[b,z,b']_{\mathrm{as}}=0$ & $L_b\circ R_{b'}=R_{b'} \circ L_b$ \\ \hline
					& $[z,b,b']_{\mathrm{as}}=0$ & $R_{bb'}=R_{b'}\circ R_b$ \\ \hline
				\end{NiceTabular}
			\end{table}
		\end{lemma}

		\begin{proposition}
			Let $(A,1_r)$ be a right-unital algebra, $x,y,z\in A$ such that the associator $[x,y,z]_{\mathrm{as}}$ is left-regular. The following statements hold:
			\begin{enumerate}[label=\textnormal{\arabic*)}]
				\item For all $b,b'\in 1_r\cdot AC_r(A),\ b\cdot b'=0$.
				\item All elements in $1_r\cdot AC_r(A)$ are 2-nilpotent.
				\item For all $b\in 1_r \cdot AC_r(A)$ and $x,y,z\in A$, $[x,y,z]_{\mathrm{as}}\cdot b\ $ is a left zero divisor.
			\end{enumerate}
		\end{proposition}
		
		\begin{theorem}
			Let $(A,1_r)$ be a domain. Then, $1_r\cdot AC_r(A)=Z(A)$ if $(A,1_r)$ is associative, and $1_r\cdot AC_r(A)=\{0\}$ otherwise. The following holds:
			\begin{enumerate}[label=\textnormal{\arabic*)},ref=\textnormal{\arabic*)}]
				\item Any right-unital hom-associative domain is either trivial or associative.
				\item Any unital hom-associative domain is either trival or associative.
				\item Any unital hom-associative division algebra is either trivial or associative.
			\end{enumerate}
		\end{theorem}
		\begin{remark}
			This expands \cref{thm:no_domains_L}: considering all possible sidings of unity, one can draw stronger conclusions.
		\end{remark}

		\subsection{Comparing one and two-sided methods}\label{ssec:comparing}
		
		Consider the subspaces $AC_l(A)$ and $AC_r(A)$ in terms of their two-sided counterpart $AC(A)$. We use the following notations:
		\begin{align*}
			AC(A)&=\{a\in Z(A)\mid Aa\subseteq N(A) \textnormal{ is a two-sided ideal}\} \\
			&\hspace{4cm}=Z(A)\cap N(A)\cap Ann_A^l([A,A,A]_{\mathrm{as}}),\\
			AC_l(A)&=\{a\in A\mid a(xy)=x(ay),\ (ax)(yz)=a((xy)z)\ \forall x,y,z\in A  \},\\
			AC_r(A)&=\{a\in A\mid (xy)a=(xa)y,\ (xy)(za)=(x(yz))a\ \forall x,y,z\in A  \},\\
			AC_n(A)&=AC_l(A)\cap AC_r(A).
		\end{align*}
		
		\begin{theorem}
			Let $(A,1_\star)$ be a one-sided unital algebra. Then,
			\begin{equation*}AC(A)=AC_n(A)=AC_l(A)\cap AC_r(A).
			\end{equation*}
			$AC(A)$ consists of all left\textnormal{(}right\textnormal{)} hom-unities that make the product of a given right\textnormal{(}left\textnormal{)}-unital algebra hom-associative.
		\end{theorem}
		\begin{proof}
			For any $a\in AC_n(A)$, commutation relations are
			\begin{align*}
				\textbf{Left unity: }(x\cdot a)\cdot y&=(x\cdot y)\cdot a\overset{x\mapsto 1_\star}{=}y\cdot a=a\cdot y\Rightarrow a\in Z(A). \\
				\textbf{Right unity: }x\cdot (a\cdot y)&=a\cdot (\cdot xy)\overset{y\mapsto 1_\star}{=}x\cdot a=a\cdot x\Rightarrow a\in Z(A).
			\end{align*}
			The association relations are as follows:
			\begin{align*}
				[a,x,y]_{\mathrm{as}}&=(a\cdot x)\cdot y-a\cdot (x\cdot y)\overset{ Z(A)}{=}(x\cdot a)\cdot y-(x\cdot y)\cdot a\overset{ AC_r(A)}{=}0,\\
				[x,a,y]_{\mathrm{as}}&=(x\cdot a)\cdot y-x\cdot (a\cdot y)\overset{ AC_n(A)}{=}(x\cdot y)\cdot a-a\cdot (x\cdot y)\overset{ Z(A)}{=}0,\\
				[x,y,a]_{\mathrm{as}}&=(x\cdot y)\cdot a-x\cdot (y\cdot a)\overset{ AC_r(A)}{=}(x\cdot a)\cdot y-x\cdot (y\cdot a)\\
				&\hspace{2.2cm}\overset{ Z(A)}{=}(x\cdot a)\cdot y-x\cdot (a\cdot y)=[x,a,y]_{\mathrm{as}}=0,
			\end{align*}
			where the small text indicates the subspace whose properties are used. It follows that $a\in Z(A)\cap N(A)$. Thus,  $AC_n(A)\subseteq Z(A)\cap N(A)$. Therefore, for all $x,y,z\in A$, we can move the parentheses as $(ax)(yz)=a(x(yz))= a((xy)z)$. It then follows that $a[x,y,z]_{\mathrm{as}}=0$, and thus
			\begin{equation*}
				AC_n(A)\subseteq Z(A)\cap N(A)\cap Ann_A^l([A,A,A]_{\mathrm{as}})=AC(A).
			\end{equation*}
			By definition, all elements of $AC(A)$ are left hom-unities (in $AC_l(A)$) and right hom-unities (in $AC_r(A)$). Thus,  $AC(A)\subseteq AC_l(A)\cap AC_r(A)= AC_n(A)$. This completes the proof.
		\end{proof}
		\begin{remark}
			Observe that in two-sided unital algebras, $AC_l(A)\cdot 1_{lr}=AC_l(A)$ and $1_{lr}\cdot AC_r(A)=AC_r(A)$. Thus, both $AC_l(A)$ and $AC_r(A)$ are subal\-gebras. The two-sided unity $1_{lr}$ makes it so that $Ann_A^l(A)=Ann_A^r(A)=\{0\}$.
		\end{remark}
		
		\begin{example}[Hom-Leibniz algebras]\label{ssec:LeibnizHom}
			By hom-Leibniz hom-algebra we understand a hom-algebra $(\mathcal{L},[\cdot,\cdot],\alpha)$ where every multiplication operator is a twisted \textit{hom-derivation} on the algebra. We distinguish between left and right hom-Leibniz algebras in a manner similar to Leibniz algebras.
			\begin{align*}
				\textbf{Left hom-Leibniz: } [\alpha(x),[y,z]]=[[x,y],\alpha(z)]+[\alpha(y),[x,z]]\\
				\textbf{Right hom-Leibniz: } [[x,y],\alpha(z)]=[[x,z],\alpha(y)]+[\alpha(x),[y,z]]
			\end{align*}
			
			Unitality in these algebras behaves differently from their untwisted counterparts because of the action of $\alpha$.
			\begin{proposition}\label{prop:LeibnizHomCrossed}
				Let $(\mathcal{L},[\cdot,\cdot],\alpha,1_\star)$ be a unital hom-Leibniz algebra that is also hom-associative. The following statements hold:
				\begin{enumerate}[label=\textnormal{\arabic*)},ref=\textnormal{\arabic*}]
					\item\label{item:HLA1} If $\mathcal{L}$ is right hom-Leibniz and $1_\star$ is a left unity, then $\alpha(1_\star)\in Ann_\mathcal{L}^r(\mathcal{L})$.
					\item\label{item:HLA2} If $\mathcal{L}$ is left hom-Leibniz and $1_\star$ is a right unity, then $\alpha(1_\star)\in Ann_\mathcal{L}^l(\mathcal{L})$.
					\item\label{item:HLA3} If $\mathcal{L}$ is right \textnormal{(}left\textnormal{)} hom-Leibniz and $1_\star$ is a right \textnormal{(}left\textnormal{)} unity, then $\alpha=0$.
					\item\label{item:HLA4}  A left-unital {\rm(}resp. right-unital\/{\rm)} hom-algebra $(\mathcal{L},[\cdot,\cdot],\alpha,1_l)$, $\alpha\neq 0$, cannot be a right hom-Leibniz \textnormal{(}resp. left hom-Leibniz\textnormal{)} and hom-associative simultaneously.
				\end{enumerate}
			\end{proposition}
			\begin{proof}
				We show the proof of \ref{item:HLA1} and \ref{item:HLA3} for right hom-Leibniz algebras. Applying the hom-Leibniz identity to three elements $\{1_l,x,y\}$, where $1_l$ is the left unity yields
				\begin{enumerate}[label=,leftmargin=*]
					\item\textbf{Left-unital: } $[x,\alpha(y)]=[[1_l,x],\alpha(y)]=[[1_l,y],\alpha(x)]+[\alpha(1_l),[x,y]]$.
				\end{enumerate}
				By unitality and hom-associativity, we have $\alpha(x)=[\alpha(1_l),x]$ for all $x\in A$ (see \cref{tab:product_relations_one}), and thus the hom-Leibniz identity becomes
				\begin{equation*}
					[x,\alpha(y)]=[y,\alpha(x)]+\alpha([x,y])\Rightarrow \alpha([x,y])=[x,\alpha(y)]-[y,\alpha(x)].
				\end{equation*}
				Contracting $x\mapsto 1_l$ yields
				\begin{equation*}
					\alpha(y)=\alpha(y)-[y,\alpha(1_l)], \text{ and thus }\ [y,\alpha(1_l)]=0, \text{ that is, } \alpha(1_l)\in Ann_\mathcal{L}^r(\mathcal{L}).
				\end{equation*}
				The right-unital case is similar to the left-unital case. Using the hom-Leibniz identity, we obtain
				\begin{enumerate}[label=,leftmargin=*]
					\item\textbf{Right-unital: } $\alpha([x,y])=[[x,y],\alpha(1_r)]=[[x,1_r],\alpha(y)]+[\alpha(x),[y,1_r]]$.
				\end{enumerate}
				By right-unitality, we have
				\begin{equation*}
					\alpha([x,y])=[x,\alpha(y)]+[\alpha(x),y].
				\end{equation*}
				Contracting $y\mapsto 1_r$ yields
				\begin{equation*}
					\alpha(x)=\alpha([x,1_r])=[x,\alpha(1_r)]+[\alpha(x),1_r]=\alpha(x)+\alpha(x),
				\end{equation*}
				and thus, $\alpha(x)=0$. This proves statement \ref{item:HLA3}. Finally, let $(\mathcal{L},[\cdot,\cdot],\alpha,1_l)$ be left-unital hom-associative, $\alpha\neq 0$. Then, $\alpha=L_a$ for $a=\alpha(1_l)\in \mathcal{L}$ (see \cref{tab:product_relations_one}). By the hom-Leibniz identity and $a$ satisfying Property \ref{item:ACL1}, we have
				\begin{align*}
					[a,[x,y]]&=[\alpha(1_l),[x,y]]=[[1_l,x],\alpha(y)]-[[1_l,y],\alpha(x)]=[x,\alpha(y)]-[y,\alpha(x)],\\
					[a,[x,y]]&\overset{\ref{item:ACL1}}{=}[x,[a,y]]=[x,\alpha(y)].
				\end{align*}
				Comparing these two expressions yields
				\begin{equation*}
					[x,\alpha(y)]-[y,\alpha(x)]=[x,\alpha(y)], \text{ and thus } [y,\alpha(x)]=0.
				\end{equation*}
				In particular, contracting $y\mapsto 1_l$ yields $\alpha(x)=0$. This contradiction proves statement \ref{item:HLA4}. The proofs for left hom-Leibniz algebras, statement \ref{item:HLA2} and the corresponding part of statements \ref{item:HLA3} and \ref{item:HLA4}, are analogous, up to using the right-unital version of Property \ref{item:ACL1} (see \cref{tab:ACR}). This completes the proof.
			\end{proof}
		\end{example}

		\section{Hom-unities in non-unital algebras}\label{sec:non-unital}
		
		In unital algebras, finding hom-associative structures is equivalent to finding hom-unity. The unity ensures a bijection between them and the compatible twisting maps. This bijection is lost in the case of non-unital algebras. Thus, for non-unital algebras it is unclear if or when the equality
		\begin{equation*}
			AC(A)=Z(A)\cap N(A)\cap Ann_A^l([A,A,A]_{\mathrm{as}})
		\end{equation*}
		holds, where $AC(A)$ is the subspace $\{a\in Z(A)\mid Aa\subseteq N(A) \textnormal{ is a two-sided ideal}\}$. We cannot guarantee that these are all the compatible two-sided hom-unities, and moreover, we cannot guarantee that all elements of $AC(A)$ are hom-unities. This section aims to construct a subspace of hom-unities using both the original definition of $AC(A)$ and the formula computed for the unital algebras.  Consider the following subspaces:
		\begin{align*}
			& HU_t(A)=\{a\in A\mid (A,L_a) \textnormal{ is hom-associative or } (A,R_a) \textnormal{ is hom-associative}\}. \\
			& HU_t^l(A)=\{a\in A\mid (A,L_a) \textnormal{ is hom-associative} \}. \\
			& HU_t^r(A)=\{a\in A\mid (A,R_a) \textnormal{ is hom-associative} \}. \\
			& HU_n(A)=Z(A)\cap N(A)\cap Ann_A^l([A,A,A]_{\mathrm{as}}). \\
			& HU_n^l(A)=Z_A(A\cdot A)\cap N^l(A)\cap N^m(A)\cap Ann_A^l([A,A,A]_{\mathrm{as}}). \\
			& HU_n^r(A)=Z_A(A\cdot A)\cap N^m(A)\cap N^r(A)\cap Ann_A^r([A,A,A]_{\mathrm{as}}). \\
			& HU_n^l(A)\cap HU_n^r(A)=Z_A(A\cdot A)\cap N(A)\cap Ann_A^l([A,A,A]_{\mathrm{as}}).
		\end{align*}
		The set $HU_t(A)$ contains all possible hom-unities of $A$, while $HU_n(A)$ is the  formulation of the subalgebra $AC(A)$ explored in \cref{sec:hom-unities-2SU}. It is then instrumental to understand how these sets interact with one another and describe explicit algebraic relations that highlight the structure of $AC(A)$, $HU_n(A)$ or their subspaces.
		\begin{proposition}\label{prop:NU_subsets}
			For any algebra $A$, the following statements hold:
			\begin{enumerate}[label=\textnormal{\arabic*)},ref=\textnormal{\arabic*}]
				\item $HU_n(A)\subseteq AC(A)$.
				\item\label{item:subsets_2} $HU_n(A)\subseteq HU_t(A)$.
				\item\label{item:subsets_3} If $A$ is two-sided unital, then $AC(A)=HU_n(A)=HU_t(A)$.
				\item $HU_n^l(A)\subseteq HU_t^l(A)$.
				\item $HU_n^r(A)\subseteq HU_t^r(A)$.
				\item $HU_n(A)\subseteq HU_n^l(A)\cap HU_n^r(A)$. Equality holds when $Z_A(A\cdot A)=Z(A)$. In particular, this is the case when $A$ is unital.
				\item $HU_n(A)$ is a subalgebra of $A$ and a two-sided ideal of $Z(A)\cap N(A)$.
			\end{enumerate}
			Moreover, if $A$ satisfies one of the following conditions
			\begin{enumerate}[label=\textnormal{\roman*.}]
				\item $A$ has no nontrivial subalgebras,
				\item $Z(A)\cap N(A)$ is simple,
			\end{enumerate}
			then one of the following properties holds:
			\begin{enumerate}[label=\textnormal{\arabic*)}]
				\setcounter{enumi}{7}
				\item $HU_n(A)=\{0\}$,
				\item $A$ is commutative associative.
			\end{enumerate}
			In addition, $HU_n(A)=A$ if and only if $A$ is commutative and associative.
		\end{proposition}
		
		\begin{proof}
			Let $a\in Z(A)\cap N(A)\cap Ann_A^l([A,A,A]_{\mathrm{as}})$.
			\begin{align*}
				[a\cdot x,y,z]_{\mathrm{as}}&=((a\cdot x)\cdot y)\cdot z-(a\cdot x)\cdot (y\cdot z)=(a\cdot (x\cdot y))\cdot z-a\cdot (x\cdot (y\cdot z))\\
				&=a\cdot ((x\cdot y)\cdot z-x\cdot (y\cdot z))=0,\\
				[x,a\cdot y,z]_{\mathrm{as}}&=(x\cdot (a\cdot y))\cdot z-x\cdot ((a\cdot y)\cdot z)=((x\cdot a)\cdot y)\cdot z-(x\cdot a)\cdot (y\cdot z)\\
				&=(a\cdot (x\cdot y))\cdot z-(a\cdot x)\cdot (y\cdot z)=a\cdot ((x\cdot y)\cdot z-x\cdot (y\cdot z))=0,\\
				[x,y,a\cdot z]_{\mathrm{as}}&=(x\cdot y)\cdot (a\cdot z)-x\cdot (y\cdot (a\cdot z))=(x\cdot y)\cdot (z\cdot a)-x\cdot (y\cdot (z\cdot a))\\
				&=((x\cdot y)\cdot z)\cdot a-x\cdot ((y\cdot z)\cdot a)=((x\cdot y)\cdot z-x\cdot (y\cdot z))a=0.
			\end{align*}
			It follows that $Aa\subseteq N(A)$. It is an ideal since, for every $x,y\in A$, we have $(x\cdot a)\cdot y=(a\cdot x)\cdot y=a\cdot (x\cdot y)\in N(A)$ and $y\cdot (x\cdot a)=(y\cdot x)\cdot a\in N(A)$. Statement \ref{item:subsets_2} follows. Statement \ref{item:subsets_3} is the topic of \cite{FreGohr:UnitCond}.
			
			Let $a\in HU_n^l(A)$. We can express the hom-associator in $(A,L_a)$ as
			\begin{multline*}
				(a\cdot x)\cdot (y\cdot z)=a\cdot (x\cdot (y\cdot z)), (x\cdot y)\cdot (a\cdot z)=((x\cdot y)\cdot a)\cdot z=(a\cdot (x\cdot y))\cdot z=a\cdot ((x\cdot y)\cdot z) \\
				\Rightarrow [x,y,z]_{\mathrm{as}}^{L_a}=a\cdot ((x\cdot y)\cdot z)-a\cdot (x\cdot (y\cdot z))=a[x,y,z]_{\mathrm{as}}=0.
			\end{multline*}
			Let $b\in HU_n^r(A)$. We can express the hom-associator in $(A,R_b)$ as
			\begin{multline*}
				(x\cdot y)\cdot (z\cdot b)=((x\cdot y)\cdot z)\cdot b, (x\cdot b)\cdot (y\cdot z)=x\cdot (b\cdot (y\cdot z))=x\cdot ((y\cdot z)\cdot b)\\
				=(x\cdot (y\cdot z))\cdot b
				\Rightarrow [x,y,z]_{\mathrm{as}}^{R_a}=((x\cdot y)\cdot z)\cdot b-(x\cdot (y\cdot z))\cdot b=[x,y,z]_{\mathrm{as}}b=0.
			\end{multline*}
			Now, observe that $Z(A)\subseteq Z_A(A\cdot A)$. Thus, $HU_n(A)\subseteq HU_n^l(A)\cap HU_n^r(A)$.
			
			It is well known that $Z(A)\cap N(A)$ is a subalgebra of any given algebra. Let $a,a'\in HU_n(A),\ b\in A,\ y\in Z(A)\cap N(A)$.
			For any element $X\in [A,A,A]_{\mathrm{as}}$, it holds that $a\cdot X=0$. Then $(a\cdot b)\cdot X=(b\cdot a)\cdot X=b\cdot (a\cdot X)=0$. In particular, $(a\cdot y)\cdot X=(y\cdot a)\cdot X=y\cdot (a\cdot X)=0$ and $(a\cdot a')\cdot X=(a'\cdot a)\cdot X=a'\cdot (a\cdot X)=0$. That is, $HU_n(A)$ is closed under products with elements of $Z(A)\cap N(A)$, and thus it is a two-sided ideal of $Z(A)\cap N(A)$, and in particular, a subalgebra of $A$.

			Finally, if $A$ has no non-trivial subalgebras or if $Z(A)\cap N(A)$ is simple, then by hypothesis, $HU_n(A)=\{0\}$ or $HU_n(A)=A$. If $HU_n(A)=A$, then $Z(A)\cap N(A)=A$, and thus $A$ is commutative and associative.
			
			Moreover, if $A$ is commutative associative, then for any $a,x,y,z\in A$ it holds that $a\cdot x\cdot y\cdot z=x\cdot y\cdot a\cdot z$, that is, $L_a(x)\cdot(y\cdot z)=(x\cdot y)\cdot L_a(z)$, or $a\in HU_n(A)$.
			
		\end{proof}
		\begin{remark}
			It is an interesting open question whether $HU_n^l(A)$ and $HU_n^r(A)$ contain all hom-unities that induce hom-associative structures over a given non-unital algebra $A$. Computing these subspaces involves finding the centralizer, partial nuclei, and associators within the algebra. This appears to be a straightforward problem based on the structure constants, as is the case with hom-associative structures over unital algebras.
		\end{remark}

		Elements of $AC(A)$ need not be two-sided hom-unities because of non-unitality. Indeed, for every $a\in AC(A)$ and all $x,y,z\in A$, $Aa\subseteq N(A)$ is a two-sided ideal, which implies the following:
		\begin{enumerate}[label=\textnormal{\arabic*)}]
			\item $a\cdot x=x\cdot a\in N(A)$,
			\item $y\cdot (a\cdot x)=(a\cdot x)\cdot y\in Aa=a A$ by centrality of $a$.
			\item In particular, $(x\cdot y)\cdot (a\cdot z)=a\cdot t^a_{xy,z}$ and $(a\cdot x)\cdot (y\cdot z)=a\cdot t^a_{x,yz}$ for some $t^a_{xy,z}$, $t^a_{x,yz}\in A$.
		\end{enumerate}
		Hom-associativity relations have the form $[x,y,z]_{\mathrm{as}}^{L_a}=L_a(t^a_{xy,z}-t^a_{x,yz})$, thus the subspace $AC(A)\cap Ann_A^l(\{t^a_{xy,z}-t^a_{x,yz}\mid x,y,z\in A\})$ contains all two-sided hom-unities that are inside $AC(A)$.
		\begin{example}
			The following cases satisfy $t^a_{xy,z}=(x\cdot y)\cdot z$, $t^a_{x,yz}=x\cdot (y\cdot z)$, and a left multiplication $L_a$ induces a hom-associative structure if and only if $a\in Ann_A^l([A,A,A]_{\mathrm{as}})$:
			\begin{enumerate}[label=\textnormal{\arabic*)}]
				\item All unital algebras. Indeed, $[x,y,z]_{\mathrm{as}}^{L_a}=a[x,y,z]_{\mathrm{as}}=a\cdot ((x\cdot y)\cdot z)-a\cdot (x\cdot (y\cdot z))$, as $a=1\cdot a\in N(A)$ (on any side).
				\item Non-unital $A$ with $a$ idempotent, or $\lambda a=a^2\in a A\subseteq N(A),\ \lambda\in\mathbb{F}$.
				\begin{align*}
					& (a\cdot x)\cdot (y\cdot z)=a\cdot (x\cdot (y\cdot z)), (x\cdot y)\cdot (a\cdot z)=((x\cdot y)\cdot a)\cdot z\\
					& \hspace{5cm} =(a\cdot (x\cdot y))\cdot z=a\cdot ((x\cdot y)\cdot z)\\
					& \Rightarrow [x,y,z]_{\mathrm{as}}^{L_a}=a\cdot ((x\cdot y)\cdot z)-a\cdot (x\cdot (y\cdot z))=a[x,y,z]_{\mathrm{as}}.
				\end{align*}
			\end{enumerate}
		\end{example}

		\begin{example}[Associative algebras, non-unital]
			
			Let $A$ be associative and not necessarily commutative. Observe that
			\begin{multicols}{2}
				$HU_n^l(A)=Z_A(A\cdot A)$,
				
				$HU_n^r(A)=Z_A(A\cdot A)$,
				
				$HU_n^l(A)\cap HU_n^r(A)=Z_A(A\cdot A)$,
				
				$HU_n(A)=Z(A)$.
			\end{multicols}
			Even without commutativity, all these hom-unities are two-sided. These need not be all hom-unities over these algebras, as the exact hom-associativity condition is $a\cdot x\cdot y\cdot z=x\cdot y\cdot a\cdot z$, or $(a\cdot x\cdot y-x\cdot y\cdot a)\cdot z=0$ for all $x,y,z\in A$.  We have
			\begin{align*}
				HU_t^l(A)&=\{a\in A\mid (a\cdot x\cdot y-x\cdot y\cdot a)\cdot z=0,\ x,y,z\in A\}, \\
				HU_t^r(A)&=\{a\in A\mid x\cdot (a\cdot y\cdot z-y\cdot z\cdot a)=0,\ x,y,z\in A\},  \\
				HU_t(A)&=HU_t^l(A)\cap HU_t^r(A),\\
				HU_n^l(A)&\cap HU_n^r(A)=\{a\in A\mid [a,x\cdot y]=0,\ x,y\in A\}.
			\end{align*}
			$HU_t^l(A)$ consists of elements such that $[a,A\cdot A]\subseteq Ann_A^l(A)$.
			$HU_t^r(A)$ consists of elements such that $[a,A\cdot A]\subseteq Ann_A^r(A)$. $HU_t(A)$ consists of elements such that $[a,A\cdot A]\subseteq Ann_A^l(A)\cup Ann_A^r(A)$.
			
			\begin{proposition}\label{prop:HUnA_assoc_no_ann}
				Let $A$ be an associative algebra.
				\begin{enumerate}[label=\textnormal{(\roman*)}]
					\item If $Ann_A^l(A)=0$, and in particular if $A$ has a right-regular element, then $$HU_t(A)=HU_n(A)=Z(A).$$
					\item If $Ann_A^r(A)=0$, and in particular if $A$ has a left-regular element, then $HU_t(A)=HU_n(A)=Z(A)$.
				\end{enumerate}
			\end{proposition}
			
		\end{example}

		\subsection{Unitalizations}\label{ssec:unitalizations}
		
		Given an algebra $A$ over a field $\mathbb{F}$, the following \textit{semi-direct product} construction over $\tilde{A}=A\oplus \mathbb{F}$ yields a two-sided unital algebra with unity $(0,1)^T$:
		\begin{align*}
			&\begin{pmatrix}
				a_1 \\ \lambda_1
			\end{pmatrix} + \begin{pmatrix}
				a_2 \\ \lambda_2
			\end{pmatrix}=\begin{pmatrix}
				a_1+a_2 \\ \lambda_1+\lambda_2
			\end{pmatrix},\\
			&\begin{pmatrix}
				a_1 \\ \lambda_1
			\end{pmatrix}\cdot\begin{pmatrix}
				a_2 \\ \lambda_2
			\end{pmatrix}=\begin{pmatrix}
				a_1a_2+a_2\lambda_1+a_1\lambda_2 \\ \lambda_1\lambda_2
			\end{pmatrix}.
		\end{align*}
		\begin{remark}
			In this section, the product in $A$ is denoted by juxtaposition. The dot represents the product in $\tilde{A}$.
		\end{remark}
		We can discuss the hom-associative structures of this construction by computing the subalgebra $AC(\tilde{A})=Z(\tilde{A})\cap N(\tilde{A})\cap Ann_{\tilde{A}}^l([\tilde{A},\tilde{A},\tilde{A}]_{\mathrm{as}})$. We compute all three subspaces as follows:
		\begin{multline*}
			\begin{pmatrix}
				b \\ \mu
			\end{pmatrix}\in Z(\tilde{A})\Leftrightarrow\forall \begin{pmatrix}
				a \\ \lambda
			\end{pmatrix}\in \tilde{A},\ \begin{pmatrix}
				ba+a\mu+b\lambda \\ \mu\lambda
			\end{pmatrix}=\begin{pmatrix}
				ab+b\lambda+a\mu \\ \lambda\mu
			\end{pmatrix}\Leftrightarrow ab=ba \\
			\Leftrightarrow b\in Z(A).
		\end{multline*}
		Now, compute the associator of three elements of $\tilde{A}$:
		\begin{multline*}
			\left(\begin{pmatrix}
				a_1 \\ \lambda_1
			\end{pmatrix}\cdot \begin{pmatrix}
				a_2 \\ \lambda_2
			\end{pmatrix}\right)\cdot \begin{pmatrix}
				a_3 \\ \lambda_3
			\end{pmatrix}=\begin{pmatrix}
				a_1a_2+a_2\lambda_1+a_1\lambda_2 \\ \lambda_1\lambda_2
			\end{pmatrix}\cdot \begin{pmatrix}
				a_3 \\ \lambda_3
			\end{pmatrix}\\
			=\begin{pmatrix}a_1(a_1a_2+a_2\lambda_1+a_1\lambda_2)a_3+a_3\lambda_1\lambda_2+(a_1a_2+a_2\lambda_1+a_1\lambda_2)\lambda_3 \\ \lambda_1\lambda_2\lambda_3
			\end{pmatrix}\\
			=\begin{pmatrix}
				(a_1a_2)a_3+\lambda_3a_1a_2+\lambda_2a_1a_3+\lambda_1a_2a_3+\lambda_2\lambda_3a_1+\lambda_1\lambda_3a_2+\lambda_1\lambda_2a_3 \\ \lambda_1\lambda_2\lambda_3
			\end{pmatrix}
		\end{multline*}
		\begin{multline*}
			\begin{pmatrix}
				a_1 \\ \lambda_1
			\end{pmatrix}\cdot\left(\begin{pmatrix}
				a_2 \\ \lambda_2
			\end{pmatrix}\cdot \begin{pmatrix}
				a_3 \\ \lambda_3
			\end{pmatrix}\right)=\begin{pmatrix}
				a_1 \\ \lambda_1
			\end{pmatrix}\cdot \begin{pmatrix}
				a_2a_3+a_3\lambda_2+a_2\lambda_3 \\ \lambda_2\lambda_3
			\end{pmatrix}\\
			=\begin{pmatrix}    a_1(a_2a_3+a_3\lambda_2+a_2\lambda_3)+(a_2a_3+a_3\lambda_2+a_2\lambda_3)\lambda_1+a_1\lambda_2\lambda_3 \\ \lambda_1\lambda_2\lambda_3
			\end{pmatrix}\\
			=\begin{pmatrix}
				a_1(a_2a_3)+\lambda_3a_1a_2+\lambda_2a_1a_3+\lambda_1a_2a_3+\lambda_2\lambda_3a_1+\lambda_1\lambda_3a_2+\lambda_1\lambda_2a_3 \\ \lambda_1\lambda_2\lambda_3
			\end{pmatrix},
		\end{multline*}
		Observe that when taking the difference, all terms vanish except $(a_1a_2)a_3$ and $a_1(a_2a_3)$.
		\begin{align*}
			\left[\begin{pmatrix}
				a_1 \\ \lambda_1
			\end{pmatrix},\begin{pmatrix}
				a_2 \\ \lambda_2
			\end{pmatrix}, \begin{pmatrix}
				a_3 \\ \lambda_3
			\end{pmatrix}\right]_{\mathrm{as}}=\begin{pmatrix}
				(a_1a_2)a_3 - a_1(a_2a_3) \\ 0
			\end{pmatrix}=\begin{pmatrix}
				[a_1,a_2,a_3]_{\mathrm{as}} \\ 0
			\end{pmatrix},
		\end{align*}
		where this quantity is zero if and only if $[a_1,a_2,a_3]_{\mathrm{as}}=0$. Thus
		\begin{equation*}
			\begin{pmatrix}
				b \\ \mu
			\end{pmatrix}\in N(\tilde{A}) \Leftrightarrow b\in N(A).
		\end{equation*}
		Finally, we search for the annihilator of all associators in $\tilde{A}$:
		\begin{multline*}
			\begin{pmatrix}
				b \\ \mu
			\end{pmatrix}\in Ann_A^l([\tilde{A},\tilde{A},\tilde{A}]_{\mathrm{as}})\Leftrightarrow \forall a_1,a_2,a_3\in A, \begin{pmatrix}
				0 \\ 0
			\end{pmatrix}=\begin{pmatrix}
				b \\ \mu
			\end{pmatrix}\cdot \begin{pmatrix}
				[a_1,a_2,a_3]_{\mathrm{as}} \\ 0
			\end{pmatrix}\\
			=\begin{pmatrix}
				b[a_1,a_2,a_3]_{\mathrm{as}}+\mu[a_1,a_2,a_3]_{\mathrm{as}} \\ 0
			\end{pmatrix}\Leftrightarrow b[a_1,a_2,a_3]_{\mathrm{as}}=-\mu[a_1,a_2,a_3]_{\mathrm{as}}\\
			\Leftrightarrow [A,A,A]_{\mathrm{as}}\subseteq Eig_{-\mu}(L_b).
		\end{multline*}
		Thus, we can describe all possible two-sided hom-unities as
		\begin{equation*}
			AC(\tilde{A})=\left\{\begin{pmatrix}
				b \\ \mu
			\end{pmatrix}\in\tilde{A}\mid b\in Z(A)\cap N(A),\ [A,A,A]_{\mathrm{as}}\subseteq Eig_{-\mu}(L_b)\right\}.
		\end{equation*}
		The map $\pi\colon AC(\tilde{A}) \rightarrow \mathbb{F}$ given by $(a,\lambda)^T \mapsto \lambda$ is a homomorphism, and moreover $\ker\pi=\begin{pmatrix}
			HU_n(A) \\ 0
		\end{pmatrix}\cong HU_n(A)$ is a two-sided ideal of $AC(\tilde{A})$.
		\begin{proposition}\label{HU_unitalization}
			Let $A$ be a non-unital algebra, $(\tilde{A},\cdot, (0,1)^T)$ its unitalization. The following statements hold:
			\begin{enumerate}[label=\textnormal{\arabic*)}]
				\item $Z(\tilde{A})\cong Z(A)\oplus \mathbb{F}$.
				\item $N(\tilde{A})\cong N(A)\oplus\mathbb{F}$.
				\item $AC(\tilde{A})=\left\{\begin{pmatrix}
					b \\ \mu
				\end{pmatrix}\in\tilde{A}\mid b\in Z(A)\cap N(A),\ [A,A,A]_{\mathrm{as}}\subseteq Eig_{-\mu}(L_b)\right\}$ is a
				
				subalgebra of $(\tilde{A},\cdot, (0,1)^T)$. Moreover, $\begin{pmatrix}
					a \\0
				\end{pmatrix}\in AC(\tilde{A})$ if and only if
				
				$a\in HU_n(A)$.
			\end{enumerate}
		\end{proposition}
		\begin{corollary}
			If $(\tilde{A},\cdot, (0,1)^T)$ is associative, then
			\begin{equation*}
				AC(\tilde{A})=Z(\tilde{A})=\left\{\begin{pmatrix}
					b \\ \mu
				\end{pmatrix}\in\tilde{A}\mid b\in Z(A)\right\}\cong Z(A)\oplus\mathbb{F}.
			\end{equation*}
		\end{corollary}
		This is natural because the unitalization process preserves the center and nucleus.
		The fact that there is a copy of $HU_n(A)$ hidden inside $AC(\tilde{A})$ as a two-sided ideal indicates that we can study certain properties using the structure of $\tilde{A}$.
		\begin{example}
			If $AC(\tilde{A})$ is simple, then either $HU_n(A)= \{0\}$ (no two-sided hom-unities of $A$ are \textnormal{good enough} elements) or $HU_n(A)\cong AC(\tilde{A})$ (hom-unities of $\tilde{A}$ uniquely determine a subalgebra of two-sided hom-unities of $A$).
		\end{example}
		
		Let $\begin{pmatrix}
			a \\ \lambda
		\end{pmatrix}, \begin{pmatrix}
			b \\ \mu
		\end{pmatrix}\in AC(\tilde{A})$. For any $X\in[A,A,A]_{\mathrm{as}}$, we have
		\begin{equation*}
			\left.\begin{array}{c}
				0=bX+\mu X= a(bX)+a\mu X=(ab)X+\mu aX  \\
				0=aX+\lambda X= b(aX)+b\lambda X=(ab)X+\lambda bX
			\end{array}\right\}\Rightarrow
			\begin{array}{l}
				\lambda bX=\mu aX,\\
				(\lambda b - \mu a)X=0.
			\end{array}
		\end{equation*}
		\begin{proposition}
			Let $A$ be non-associative, $(\tilde{A},\cdot,(0,1)^T)$ its unitalization. For all $\begin{pmatrix}
				a \\ \lambda
			\end{pmatrix},\begin{pmatrix}
				b \\ \mu
			\end{pmatrix}\in AC(\tilde{A})$ the following statements hold:
			\begin{enumerate}[label=\textnormal{\arabic*)}]
				\item $(\lambda b-\mu a)[A,A,A]_{\mathrm{as}}=\{0\}$. In particular, $\lambda b-\mu a\in HU_n(A)$.
				\item If $\lambda=\mu$, then $b-a\in HU_n(A)$.
				\item If $a=b$, then either $a=0$ or $\lambda=\mu$.
			\end{enumerate}
		\end{proposition}
		\begin{proof}
			From $(\lambda b - \mu a)X=0$, we obtain $(\lambda b-\mu a)[A,A,A]_{\mathrm{as}}=\{0\}$. This can be read $(\lambda b-\mu a)\in Ann_A^l([A,A,A]_{\mathrm{as}})$, and thus $(\lambda b-\mu a)\in HU_n(A)$.
			
			If $a=b\in HU_n(A)$, then $\lambda=\mu=0$ by \cref{HU_unitalization}. If $a=b\notin HU_n(A)$, then $\{0\}=(\lambda a-\mu a)[A,A,A]_{\mathrm{as}}=(\lambda-\mu)a[A,A,A]_{\mathrm{as}}$. The conclusion follows from $a\notin HU_n(A)$.
		\end{proof}
		
		$AC(\tilde{A})$ being a subalgebra of $(\tilde{A},\cdot, (0,1)^T)$ indicates that it is closed by products. The product of any two elements $(a,\lambda)^T,(b,\mu)^T\in AC(\tilde{A})$ is then in $AC(\tilde{A})$, that is,
		\begin{align*}
			\begin{pmatrix}
				a \\ \lambda
			\end{pmatrix}\cdot \begin{pmatrix}
				b \\ \mu
			\end{pmatrix}=\begin{pmatrix}
				ab+\lambda b+\mu a \\ \lambda\mu
			\end{pmatrix}\in AC(\tilde{A})\Rightarrow \left\{\begin{array}{c}
				ab+\lambda b+\mu a\in Z(A)\cap N(A), \\
				\left[A,A,A\right]_{\mathrm{as}}\subseteq Eig_{-\lambda\mu}(L_{ab+\lambda b+\mu a}).
			\end{array} \right.
		\end{align*}

			 \begin{example}[Skew-symmetric algebras]\label{ex:skewsym}
           In any skew-sym\-metric algebra with left unity $1_1$ element, $a=[1_l,a]=[1_l,1_l]a=-[[1_l,1_l],a]=-[1_l,a]=-a\Rightarrow 2a=0$, and similarly, the same conclusion $a=-a, 2a=0 $ holds in any skew-sym\-metric algebra with right unity.
            Since $2a=0\Rightarrow a=0$ for all elements in any algebra over fields of characteristic different from $2$, the only such algebra is the trivial zero algebra $L=\{0\}$. For algebras over fields of characteristic $2$, this conclusion does not hold because $a=-a, 2a=0$ holds for all $a\neq 0$.
            This implies that a skew-symmetric algebra over a field of characteristic $2$ is  also symmetric because $x=-x \Rightarrow [x,y]=-[y,x]=[y,x]$, while it does not necessarily yield a zero product.
            Unlike algebras over fields of other characteristics, in characteristic $2$, a non-zero algebra can be both skew-symmetric and unital. For example, any field $F$ of characteristic $2$ (such as $\mathbb{F}_{2}$) is a nonzero, unital, skew-symmetric algebra over itself with multiplication that is commutative and satisfies $xy=yx=-yx$.
            \end{example}

            \begin{example}[Leibniz algebras]\label{ssec:Leibniz}
			Leibniz algebras $(\mathcal{L},[\cdot,\cdot])$ are those where every product operator is a derivation on the algebra. We distinguish between left and right Leibniz algebras according to which products are derivations.
			\begin{align*}
				\textbf{Left Leibniz: } [a,[b,c]]=[[a,b],c]+[b,[a,c]]\\
				\textbf{Right Leibniz: } [[a,b],c]=[[a,c],b]+[a,[b,c]]
			\end{align*}
            Leibniz algebras with non-trivial, that is, non-zero multiplication,
            cannot be unital, as any one-sided unity collapses the product:
                \begin{align*}
            		&\textbf{Left-unital: }\\
            		& [b,c]=[[1_l,b],c]=[[1_l,c],b]+[1_l,[b,c]]=[c,b]+[b,c] \Longrightarrow [c,b]=0,\\
            		&\textbf{Right-unital: }\\
            		& [a,b]=[[a,b],1_r]=[[a,1_r],b]+[a,[b,1_r]]=[a,b]+[a,b]
            		\Rightarrow [a,b]=0.
            	\end{align*}
            Note that for Leibniz algebras, this incompatibility of the existence of a non-zero unity with the product being non-zero holds independently of the characteristic of the field, while in skew-symmetric algebras, it is guaranteed for fields with characteristics different from $2$, while it often fails for fields with characteristic $2$. The incompatibility of the existence of a non-zero unity with the product being non-zero holds in Lie algebras over fields of any characteristic, as they are a subclass of Leibniz algebras.
            \end{example}

			Therefore, for Leibniz algebras, one must rely on the unitalization process: $HU_n(\mathcal{L})$ is a subspace of $AC(\mathcal{L})$ which we can compute.

            \begin{remark}
            In the literature, for example in {\rm \cite{Man:LowDimLeib}}, the \textit{centers} of a Leibniz algebra are given by $Z_l(\mathcal{L})=\{a\in \mathcal{L}\mid [a,\mathcal{L}]=\{0\}\}$, $Z_r(\mathcal{\mathcal{L}})=\{a\in \mathcal{L}\mid [\mathcal{L},a]=\{0\}\}$ and $Z(\mathcal{L})=Z_l(\mathcal{L})\cap Z_r(\mathcal{L})$.
            In context of our article, we emphasise that these sets are the \textnormal{annihilators} $Z_l(\mathcal{L})=Ann_\mathcal{L}^l(\mathcal{L})$,  $Z_r(\mathcal{\mathcal{L}})=Ann_\mathcal{L}^r(\mathcal{L})$ and $Z(\mathcal{L})=Ann_\mathcal{L}(\mathcal{L})$. In order to avoid notational overlap with the commutative center we use this for the classical centers of Leibniz algebras, while we denote the commutative center by $C(\mathcal{L})=\{a\in\mathcal{L}\mid [a,x]=[x,a]\ \forall x\in\mathcal{L}\}$.
            \end{remark}

            \begin{proposition}\label{prop:LeibHomUnities}
				Let $(\mathcal{L},[\cdot,\cdot])$ be a Leibniz algebra. Then $(\mathcal{L},[\cdot,\cdot],L_a)$ is hom-associative for all $a$ in the subspace
				\begin{equation*}
					HU_n(\mathcal{L})=C(\mathcal{L})\cap Ann_\mathcal{L}^l([\mathcal{L},\mathcal{L}])=C(\mathcal{L})\cap Ann_\mathcal{L}^r([\mathcal{L},\mathcal{L}])=C(\mathcal{L})\cap Ann_\mathcal{L}([\mathcal{L},\mathcal{L}]).
				\end{equation*}
				Moreover, every element in $HU_n(\mathcal{L})$ is 3-nilpotent.
			\end{proposition}
			
			\begin{proof}
				Compute the associator of three elements $a,b,c\in\mathcal{L}$ using the Leibniz rule:
				\begin{align*}
					\textbf{Left: }[a,[b,c]]=[[a,b],c]+[b,[a,c]]\Rightarrow [a,b,c]_{\mathrm{as}}=-[b,[a,c]].\\
					\textbf{Right: }[[a,b],c]=[[a,c],b]+[a,[b,c]]\Rightarrow [a,b,c]_{\mathrm{as}}=[[a,c],b].
				\end{align*}
				It follows that $[\mathcal{L},\mathcal{L},\mathcal{L}]_{\mathrm{as}}=[[\mathcal{L},\mathcal{L}],\mathcal{L}]$ (left) and $[\mathcal{L},\mathcal{L},\mathcal{L}]_{\mathrm{as}}=[\mathcal{L},[\mathcal{L},\mathcal{L}]]$ (right).
				
				Any $a\in N(\mathcal{L})$ satisfies, for all $b,c\in\mathcal{L}$, the following relations if $\mathcal{L}$ is left Leibniz:
				\begin{align*}
					\left.\begin{array}{c}
						0=[c,b,a]_{\mathrm{as}}=[[c,b],a]-[c,[b,a]]=-[b,[c,a]]  \\
						0=[b,c,a]_{\mathrm{as}}=[[b,c],a]-[b,[c,a]]=-[c,[b,a]]
					\end{array}\right\}&\Rightarrow\ \ [[b,c],a]=[c,[b,a]]=0\\
					\left.\begin{array}{c}
						0=[a,b,c]_{\mathrm{as}}=[[a,b],c]-[a,[b,c]]=-[b,[a,c]]  \\
						0=[b,a,c]_{\mathrm{as}}=[[b,a],c]-[b,[a,c]]=-[c,[a,b]]  \\
					\end{array}\right\}&\Rightarrow \left\{\begin{matrix}
						[[a,b],c]=[[b,a],c]=0\\
						[a,[b,c]]=[b,[a,c]]=0
					\end{matrix} \right.
				\end{align*}
				and the following relations if $\mathcal{L}$ is right-Leibniz:
				\begin{align*}
					\left.\begin{array}{c}
						0=[a,b,c]_{\mathrm{as}}=[[a,b],c]-[a,[b,c]]=[[a,c],b]  \\
						0=[a,c,b]_{\mathrm{as}}=[[a,c],b]-[a,[c,b]]=[[a,b],c]
					\end{array}\right\}&\Rightarrow\ \ [a,[b,c]]=[[a,c],b]=0\\
					\left.\begin{array}{c}
						0=[c,b,a]_{\mathrm{as}}=[[c,b],a]-[c,[b,a]]=[[c,a],b]  \\
						0=[c,a,b]_{\mathrm{as}}=[[c,a],b]-[c,[a,b]]=[[c,b],a]
					\end{array}\right\}&\Rightarrow \left\{\begin{matrix}
						[c,[b,a]]=[c,[a,b]]=0\\
						[[c,b],a]=[[c,a],b]=0
					\end{matrix} \right.
				\end{align*}
				Every triple product involving $a$ is zero. In particular, $a$ commutes with any $[b,c]$ and is 3-nilpotent. That is, $N(\mathcal{L})$, and in particular $C(\mathcal{L})\cap N(\mathcal{L})$, is a subspace of the subspace of 3-nilpotent elements of $\mathcal{L}$. By the Lebniz identity,
				$[\mathcal{L},\mathcal{L},\mathcal{L}]_{\mathrm{as}}=[[\mathcal{L},\mathcal{L}],\mathcal{L}]\subseteq [\mathcal{L},\mathcal{L}]$ and thus  $a\in N(\mathcal{L})\Rightarrow a\in Ann_\mathcal{L}^l([\mathcal{L},\mathcal{L}])\subseteq Ann_\mathcal{L}^l([[\mathcal{L},\mathcal{L}],\mathcal{L}])= Ann_\mathcal{L}^l([\mathcal{L},\mathcal{L},\mathcal{L}]_{\mathrm{as}})$. This indicates that $C(\mathcal{L})\cap N(\mathcal{L})$ cancels $[\mathcal{L},\mathcal{L}]$ on both sides by centrality, so we can use left, right, or two-sided annihilator notation indistinctly. This completes the proof.
			\end{proof}				

     Proposition \ref{prop:LeibHomUnities} also illustrates that the study of hom-associative structures is linked to the existence of $3$-nilpoitent elements in the algebra.
		
		\begin{example}[Hom-Leibniz algebras, non-unital]		
			The discussion of hom-associ\-ative structures in hom-Leibniz algebras, as seen in Example \ref{ssec:LeibnizHom} can continue in two directions:
			\begin{enumerate}
				\item Compute $AC_l(A)\cdot 1_l$ and examine its existence.
				\item Examine non-unital hom-Leibniz algebras.
			\end{enumerate}
			
			We start by examining non-unital hom-Leibniz algebras twisted by multiplication operators: if $(\mathcal{L},[\cdot,\cdot],\alpha)$ is non-unital, but $\alpha=L_a$ is a multiplication operator such that $a=[a,w]=\alpha(w)$ for some $w\in\mathcal{L}$ (for example, if $a$ is idempotent), then
			\begin{align*}
				[[w,x],\alpha(y)]&=[[w,y],\alpha(x)]+[\alpha(w),[x,y]]=[[w,y],\alpha(x)]+\alpha([x,y]),\\
				\alpha([x,y])&=[[w,x],\alpha(y)]-[[w,y],\alpha(x)].
			\end{align*}
			Observe that then
			\begin{multline*}
				\alpha([y,x])=[[w,y],\alpha(x)]-[[w,x],\alpha(y)]\\
				=-\big([[w,x],\alpha(y)]-[[w,y],\alpha(x)]\big)=-\alpha([x,y]).
			\end{multline*}
			The product $\alpha\circ[\cdot,\cdot]$ is skew-symmetric, and in particular, $\alpha([x,x])=0$. We then have that $(\mathcal{L},L_a\circ [\cdot,\cdot],L_a\circ L_a)$ is a hom-Lie algebra if $L_a$ is multiplicative. This hom-Lie algebra is multiplicative because it is a Yau twist.
			
			\begin{proposition}\label{prop:HLeibYauTwist}
				Let $a\in\mathcal{L}$ be such that $a=[a,w]$ for some $w\in\mathcal{L}$. The Yau twist of a multiplicative right hom-Leibniz algebra $(\mathcal{L},[\cdot,\cdot],L_a)$ by $L_a$, that is, the algebra $(\mathcal{L},L_a\circ [\cdot,\cdot],L_a\circ L_a)$, is a multiplicative hom-Lie algebra.
				
				Let $a\in\mathcal{L}$ be such that $a=[w,a]$ for some $w\in\mathcal{L}$. The Yau twist of a multiplicative left hom-Leibniz algebra $(\mathcal{L},[\cdot,\cdot],R_a)$ by $R_a$, that is, the algebra $(\mathcal{L},R_a\circ[\cdot,\cdot],R_a\circ R_a)$, is a multiplicative hom-Lie algebra.
			\end{proposition}

			Consider a left-unital, right hom-Leibniz algebra $(\mathcal{L},[\cdot,\cdot],\alpha,1_l)$, and let $\beta=L_b$ be HA-compatible, that is, $(\mathcal{L},[\cdot,\cdot],\beta,1_l)$ is hom-associative.
			\begin{align*}
				\textbf{Right hom-Leibniz: }& [x,y,z]_{\mathrm{as}}^\alpha=[[x,z],\alpha(y)], \\
				\textbf{hom-associative: }& [x,y,z]_{\mathrm{as}}^\beta=0.
			\end{align*}
			This indicates that there exist $x,y,z\in\mathcal{L}$ such that $[\beta(x),[z,y]]=[[x,z],\beta(y)]\neq 0$. Otherwise, $(\mathcal{L},[\cdot,\cdot],\beta,1_l)$ would be a right hom-Leibniz algebra, and thus $b=0$ as previously established. Additionally, $\alpha$ can be a multiplication operator, as long as, if it is left,  the factor is not in $AC_l(A)$. Otherwise, $(\mathcal{L},[\cdot,\cdot],\alpha,1_l)$ would also be hom-associative, and thus $\alpha=0$ as we have established.
			
		\end{example}
		
		\begin{example}[Commutative algebras, non-unital]
			Let $A$ be a commutative algebra. If we look for hom-unities, all of them are immediately two-sided, and moreover, any element that cancels all associators does so on both sides. We can then describe the subspace $HU_n(A)$ of hom-unities through the unitali\-zation process: it is the subspace of
			$$AC(\tilde{A})=\left\{\begin{pmatrix}
				a \\ \mu
			\end{pmatrix}\mid a\in N(A),\ [A,A,A]_{\mathrm{as}}\subseteq Eig_{-\mu}(L_a) \right\}$$
			corresponding to $\mu=0$, or $N(A)\cap Ann_A^l([A,A,A]_{\mathrm{as}})$.
			
			First, observe that these algebras are always flexible:
			\begin{equation*}
				[x,y,x]_{\mathrm{as}}=(xy)x-x(yx)=(xy)x-(yx)x=(xy)x-(xy)x=0.
			\end{equation*}
			
			The associator of three elements is
			\begin{equation*}
				[x,y,z]_{\mathrm{as}}=(xy)z-x(yz)=z(xy)-x(zy)=[L_z,L_x](y),
			\end{equation*}
			and thus $[A,A,A]_{\mathrm{as}}=\{[L_z,L_x](y)\mid x,y,z\in A\}$. Using the unitalization process, we obtain a subspace of hom-unities given by the eigenspaces that contain all associators. In terms of $L_a$, this can be expressed as
			\begin{align*}
				(L_a\circ [L_z,L_x])(y)=-\mu([L_z,L_x](y))\Rightarrow (L_a+\mu\cdot \mathrm{id}_A)([L_z,L_x](y))=0.
			\end{align*}
			Now, $HU_n(A)$ is isomorphic to the subspace of $AC(\tilde{A})$ with $\mu=0$, that is
			\begin{equation*}
				HU_n(A)=\{a\in N(A)\mid L_a\circ[L_z,L_x]=0\ \forall x,z\in A \},
			\end{equation*}
			or explicitly in terms of elements as
			\begin{equation*}
				HU_n(A)=\{a\in N(A)\mid a(x(yz))=a(z(yx))\ \forall x,y,z\in A \}.
			\end{equation*}
		\end{example}

	\end{document}